\documentclass[10pt]{amsart}
\usepackage{a4wide,color}
\usepackage{amsmath,amssymb}
\usepackage[active]{srcltx}
\usepackage{ulem}

\allowdisplaybreaks
\let\pa\partial

\newtheorem{theorem}{Theorem}
\newtheorem{lemma}[theorem]{Lemma}
\newtheorem{proposition}[theorem]{Proposition}

\let\ga=\gamma
\let\de=\delta

\let\pa=\partial

\let\om=\omega

\begin{document}
\title[Boltzmann equation]{Quantitative Pointwise Estimate of the Solution
of the Linearized Boltzmann Equation}
\author{Yu-Chu Lin}
\address{Yu-Chu Lin, Department of Mathematics, National Cheng Kung
University, Tainan, Taiwan}
\email{yuchu@mail.ncku.edu.tw }
\author{Haitao Wang}
\address{Haitao Wang, Institute of Natural Sciences and School of
Mathematical Sciences, Shanghai Jiao Tong University, Shanghai, China}
\email{haitaowang.math@gmail.com}
\author{Kung-Chien Wu}
\address{Kung-Chien Wu, Department of Mathematics, National Cheng Kung
University, Tainan, Taiwan and National Center for Theoretical Sciences,
National Taiwan University, Taipei, Taiwan}
\email{kungchienwu@gmail.com}
\thanks{The first author is supported by the Ministry of Science and
Technology under the grant MOST 105-2115-M-006-002-. The third author is
supported by the Ministry of Science and Technology under the grant
104-2628-M-006-003-MY4 and National Center for Theoretical Sciences.}

\begin{abstract}
We study the quantitative pointwise behavior of the solutions of the
linearized Boltzmann equation for hard potentials, Maxwellian molecules and
soft potentials, with Grad's angular cutoff assumption. More precisely, for
solutions inside the finite Mach number region, we obtain the pointwise
fluid structure for hard potentials and Maxwellian molecules, and optimal
time decay in the fluid part and sub-exponential time decay in the non-fluid
part for soft potentials. For solutions outside the finite Mach number
region, we obtain sub-exponential decay in the space variable. The singular
wave estimate, regularization estimate and refined weighted energy estimate
play important roles in this paper. Our results largely extend the classical
results of Liu-Yu \cite{[LiuYu], [LiuYu2], [LiuYu1]} and Lee-Liu-Yu \cite%
{[LeeLiuYu]} to hard and soft potentials by imposing suitable exponential
velocity weight on the initial condition.
\end{abstract}

\maketitle





\section{Introduction}

\subsection{The models}

In this paper, we consider the following Boltzmann equation:
\begin{equation}
\left\{
\begin{array}{l}
\displaystyle\pa_{t}F+\xi \cdot \nabla _{x}F=Q(F,F)\,, \\[4mm]
\displaystyle F(0,x,\xi )=F_{0}(x,\xi )\,,%
\end{array}%
\right.  \label{bot.1.a}
\end{equation}%
where $F(t,x,\xi )$ is the distribution function for the particles at time $%
t>0$, position $x=(x_{1},x_{2},x_{3})\in {\mathbb{R}}^{3}$ and microscopic
velocity $\xi =(\xi _{1},\xi _{2},\xi _{3})\in {\mathbb{R}}^{3}$. The
left-hand side of this equation models the transport of particles and the
operator on the right-hand side models the effect of collisions on the
transport
\begin{equation*}
Q(F,G)=\int_{{\mathbb{R}}^{3}\times S^{2}}|\xi -\xi _{\ast }|^{\gamma
}B(\vartheta )\left\{ F(\xi _{\ast }^{\prime })G(\xi ^{\prime })-F(\xi
_{\ast })G(\xi )\right\} d\xi _{\ast }d\om\,.
\end{equation*}%
In this paper, we consider the Maxwellian molecules ($\gamma =0$), hard
potentials ($0<\gamma <1$) and soft potentials ($-2<\gamma <0$); and $%
B(\vartheta )$ satisfies the Grad cutoff assumption
\begin{equation*}
0<B(\vartheta )\leq C|\cos \vartheta |\,,
\end{equation*}%
for some constant $C>0$. Moreover, the post-collisional velocities satisfy
\begin{equation*}
\xi ^{\prime }=\xi -[(\xi -\xi _{\ast })\cdot \om]\om\,,\quad \xi _{\ast
}^{\prime }=\xi +[(\xi -\xi _{\ast })\cdot \om]\om\,,
\end{equation*}%
and $\vartheta $ is defined by
\begin{equation*}
\cos \vartheta =\frac{|(\xi -\xi _{\ast })\cdot \om|}{|\xi -\xi _{\ast }|}\,.
\end{equation*}%
It is well known that the Maxwellians are steady state solutions to the
Boltzmann equation. Thus, it is natural to linearize the Boltzmann equation (%
\ref{bot.1.a}) around a global Maxwellian
\begin{equation*}
\mathcal{M}(\xi )=\frac{1}{(2\pi )^{3/2}}\exp \Big(\frac{-|\xi |^{2}}{2}\Big)%
\,,
\end{equation*}%
with the standard perturbation $f(t,x,\xi )$ to $\mathcal{M}$ as
\begin{equation*}
F=\mathcal{M}+\mathcal{M}^{1/2}f\,.
\end{equation*}%
After substituting into (\ref{bot.1.a}) and dropping the nonlinear term, we
then have the linearized Boltzmann equation
\begin{equation}
\left\{
\begin{array}{l}
\displaystyle\pa_{t}f+\xi \cdot \nabla _{x}f=\mathcal{M}^{-1/2}\left[ Q(%
\mathcal{M},\mathcal{M}^{1/2}f)+Q(\mathcal{M}^{1/2}f,\mathcal{M})\right]
=Lf\,, \\[4mm]
\displaystyle f(0,x,\xi )=f_{0}(x,\xi )\,.%
\end{array}%
\right.  \label{bot.1.d}
\end{equation}

It is well-known that the null space of $L$ is a five-dimensional vector
space with the orthonormal basis $\{\chi _{i}\}_{i=0}^{4}$, where
\begin{equation*}
Ker(L)=\left\{ \chi _{0},\chi _{i},\chi _{4}\right\} =\left\{ \mathcal{M}%
^{1/2},\ \xi _{i}\mathcal{M}^{1/2},\ \frac{1}{\sqrt{6}}(|\xi |^{2}-3)%
\mathcal{M}^{1/2}\right\} \,,\quad i=1,\ 2,\ 3\,.
\end{equation*}%
Based on this property, we can introduce the Macro-Micro decomposition: let $%
\mathrm{P}_{0}$ be the orthogonal projection with respect to the $L_{\xi
}^{2}$ inner product onto $\mathrm{Ker}(L)$, and $\mathrm{P}_{1}\equiv
\mathrm{Id}-\mathrm{P}_{0}$.

\subsection{Main results}

Before the presentation of the main theorem, let us define some notations in
this paper. We denote $\left\langle \xi \right\rangle ^{s}=(1+|\xi
|^{2})^{s/2}$, $s\in {\mathbb{R}}$. For the microscopic variable $\xi $, we
denote
\begin{equation*}
|g|_{L_{\xi }^{2}}=\Big(\int_{{\mathbb{R}}^{3}}|g|^{2}d\xi \Big)%
^{1/2}\,,\quad |g|_{L_{\xi }^{\infty }}=\sup_{\xi \in {\mathbb{R}}%
^{3}}|g(\xi )|\,,
\end{equation*}%
and the weighted norms can be defined by
\begin{equation*}
|g|_{L_{\xi }^{2}(m)}=\Big(\int_{{\mathbb{R}}^{3}}|g|^{2}md\xi \Big)%
^{1/2}\,,\quad |g|_{L_{\xi }^{\infty }(m)}=\sup_{\xi \in {\mathbb{R}}%
^{3}}\left\{ |g(\xi )|m\right\} \,.
\end{equation*}%
The $L_{\xi }^{2}$ inner product in ${\mathbb{R}}^{3}$ will be denoted by $%
\big<\cdot ,\cdot \big>_{\xi }$,
\begin{equation*}
\left\langle f,g\right\rangle _{\xi }=\int f(\xi )\overline{g(\xi )}d\xi .
\end{equation*}%
For the space variable $x$, we have similar notations. In fact,
\begin{equation*}
|g|_{L_{x}^{2}}=\Big(\int_{{\mathbb{R}^{3}}}|g|^{2}dx\Big)^{1/2}\,,\quad
|g|_{L_{x}^{\infty }}=\sup_{x\in {\mathbb{R}^{3}}}|g(x)|\,.
\end{equation*}%
The standard vector product will be denoted by $(a,b)$ or $a\cdot b$ for any
vectors $a,b\in {\mathbb{R}}^{3}$. For the Boltzmann equation, the natural
norm in $\xi $ is $|\cdot |_{L_{\sigma }^{2}}$, which is defined by
\begin{equation*}
|g|_{L_{\sigma }^{2}}^{2}=|\left\langle \xi \right\rangle ^{\frac{\ga}{2}%
}g|_{L_{\xi }^{2}}^{2}\,.
\end{equation*}%
Moreover, we define
\begin{equation*}
\Vert g\Vert _{L^{2}}^{2}=\int_{{\mathbb{R}^{3}}}|g|_{L_{\xi
}^{2}}^{2}dx\,,\quad \Vert g\Vert _{L^{2}(m)}^{2}=\int_{{\mathbb{R}^{3}}%
}|g|_{L_{\xi }^{2}(m)}^{2}dx\,,
\end{equation*}%
and
\begin{equation*}
\Vert g\Vert _{L_{x}^{\infty }L_{\xi }^{\infty }(m)}=\sup_{(x,\xi )\in {%
\mathbb{R}^{6}}}\left\{ |g(x,\xi )|m\right\} \,,\quad \Vert g\Vert
_{L_{x}^{1}L_{\xi }^{2}(m)}=\int_{{\mathbb{R}^{3}}}|g|_{L_{\xi }^{2}(m)}dx\,.
\end{equation*}%
Finally, we define the high order Sobolev norms: let $s_{1},s_{2}\in {%
\mathbb{N}}$ and let $\alpha _{1},\alpha _{2}$ be any multi-indexes with $%
|\alpha _{1}|\leq s_{1}$ and $|\alpha _{2}|\leq s_{2}$,
\begin{equation*}
\left\Vert g\right\Vert _{H_{x}^{s_{1}}L_{\xi }^{2}(m)}=\sum_{|\alpha
_{1}|\leq s_{1}}\left\Vert \pa_{x}^{\alpha _{1}}g\right\Vert
_{L^{2}(m)}\,,\quad \left\Vert g\right\Vert _{L_{x}^{2}H_{\xi
}^{s_{2}}(m)}=\sum_{|\alpha _{2}|\leq s_{2}}\| \pa_{\xi }^{\alpha
_{2}}g\|_{L^{2}(m)}\,.
\end{equation*}

The domain decomposition plays an important role in our analysis, hence we
need to define a cut-off function $\chi:{\mathbb{R}}\rightarrow{\mathbb{R}}$%
, which is a smooth non-increasing function, $\chi(s)=1$ for $s\leq1$, $%
\chi(s)=0$ for $s\geq2$ and $0\leq\chi\leq1$. Moreover, we define $\chi
_{R}(s)=\chi(s/R)$.

For simplicity of notations, hereafter, we abbreviate \textquotedblleft {\ $%
\leq C$} \textquotedblright\ to \textquotedblleft {\ $\lesssim $ }%
\textquotedblright , where $C$ is a positive constant depending only on
fixed numbers.

The precise description of our main result is as follows.

\begin{theorem}
\label{thm:main} Let $f$ be a solution to \eqref{bot.1.d} with initial data $%
f_{0}$ compactly supported in $x$-variable and bounded in the weighted $\xi $%
-space
\begin{equation*}
f_{0}(x,\xi )\equiv 0,\;\text{ for }\left\vert x\right\vert \geq 1.
\end{equation*}%
There exists a positive constant $M$ such that the following hold:

\begin{enumerate}
\item As $0\leq \gamma <1$, for any given positive integer $N$, any given $%
0<p\leq 2$ and any sufficiently small $\alpha ,\ \epsilon >0$, there exist
positive constants $C$, $C_{N}$, $c_{0}$ and $c_{\epsilon }$ such that $f$
satisfies

\begin{enumerate}
\item For $\left\langle x\right\rangle \leq 2Mt$,
\begin{equation*}
\left\vert f(t,x,\cdot )\right\vert _{L_{\xi }^{2}}\leq C_{N}\left[
\begin{array}{l}
\left( 1+t\right) ^{-2}\left( 1+\frac{\left( \left\vert x\right\vert -%
\mathbf{v}t\right) ^{2}}{1+t}\right) ^{-N}+\left( 1+t\right) ^{-3/2}\left( 1+%
\frac{\left\vert x\right\vert ^{2}}{1+t}\right) ^{-N} \\[2mm]
+\mathbf{1}_{\{\left\vert x\right\vert \leq \mathbf{v}t\}}\left( 1+t\right)
^{-3/2}\left( 1+\frac{\left\vert x\right\vert ^{2}}{1+t}\right) ^{-3/2} \\%
[2mm]
+e^{-c_{0}\left( t+\alpha ^{\frac{1-\gamma }{p+1-\gamma }}\left\vert
x\right\vert ^{\frac{p}{p+1-\gamma }}\right) }+e^{-t/C}%
\end{array}%
\right] |||f_{0}|||\,.
\end{equation*}

\item For $\left\langle x\right\rangle \geq 2Mt$,
\begin{equation*}
\left\vert f(t,x,\cdot )\right\vert _{L_{\xi }^{2}}\leq C\left(
e^{-c_{0}\left( t+\alpha ^{\frac{1-\gamma }{p+1-\gamma }}\left\vert
x\right\vert ^{\frac{p}{p+1-\gamma }}\right) }+t^{5}e^{-c_{\epsilon
}(\left\langle x\right\rangle +t)^{\frac{p}{p+1-\gamma }}}\right)
|||f_{0}|||\,.
\end{equation*}
\end{enumerate}

\item As $-2<\gamma <0$, for any given $0<p\leq 2$ and any sufficiently
small $\alpha ,\ \epsilon >0$, there exist positive constants $C$, $c$, $%
c_{0}$ and $c_{\epsilon }$ such that $f$ satisfies

\begin{enumerate}
\item For $\left\langle x\right\rangle \leq 2Mt$,
\begin{equation*}
\left\vert f(t,x,\cdot )\right\vert _{L_{\xi }^{2}}\leq C\left[
\begin{array}{l}
\left( 1+t\right) ^{-3/2}+e^{-c\alpha ^{\frac{-\gamma }{p-\gamma }}t^{\frac{p%
}{p-\gamma }}} \\[2mm]
+e^{-c_{0}\left( \alpha ^{\frac{-\gamma }{p-\gamma }}t^{\frac{p}{p-\gamma }%
}+\alpha ^{\frac{1-\gamma }{p+1-\gamma }}\left\vert x\right\vert ^{\frac{p}{%
p+1-\gamma }}\right) }%
\end{array}%
\right] |||f_{0}|||\,.
\end{equation*}

\item For $\left\langle x\right\rangle \geq 2Mt$,
\begin{equation*}
\left\vert f(t,x,\cdot )\right\vert _{L_{\xi }^{2}}\leq C\left[
\begin{array}{c}
e^{-c_{0}\left( \alpha ^{\frac{-\gamma }{p-\gamma }}t^{\frac{p}{p-\gamma }%
}+\alpha ^{\frac{1-\gamma }{p+1-\gamma }}\left\vert x\right\vert ^{\frac{p}{%
p+1-\gamma }}\right) } \\
+t^{5}e^{-c_{\epsilon }(\left\langle x\right\rangle +t)^{\frac{p}{p+1-\gamma
}}}%
\end{array}%
\right] |||f_{0}|||\,.
\end{equation*}
\end{enumerate}
\end{enumerate}

Here $\mathbf{1}_{\{\cdot \}}$ is the indicator function and%
\begin{equation*}
|||f_{0}|||\equiv \max \left\{ \left\Vert f_{0}\right\Vert _{L^{2}\left(
e^{7\alpha \left\vert \xi \right\vert ^{p}}\right) },\left\Vert
f_{0}\right\Vert _{L_{x}^{1}L_{\xi }^{2}},\left\Vert f_{0}\right\Vert
_{L^{2}\left( e^{\epsilon \left\vert \xi \right\vert ^{p}}\right)
},\left\Vert f_{0}\right\Vert _{L_{x}^{\infty }L_{\xi }^{\infty }\left(
e^{8\alpha \left\vert \xi \right\vert ^{p}}\right) }\right\} .
\end{equation*}
The constant $\mathbf{v}=\sqrt{5/3}$ is the sound speed associated with the
normalized global Maxwellian.
\end{theorem}

\subsection{Method of proof}

The pointwise behavior of the solutions of the linearized Boltzmann equation
has been investigated in \cite{[LiuYu], [LiuYu2], [LiuYu1]} for the hard
sphere case and \cite{[LeeLiuYu]} for the hard potential case. On the other
hand, we aware that a stronger velocity weight yields not only faster time
decay (see Caflisch \cite{[Caflisch]} and Strain-Guo \cite{[Strain-Guo]}),
but also space decay (see Golse-Poupand \cite{[GolsePoupand]}). In this
regard, we are interested in the pointwise behavior of the solution of the
linearized Boltzmann equation with hard, Maxwellian and soft potentials,
under an exponential velocity weight on the initial condition.

In this paper, the same as in \cite{[LeeLiuYu], [LiuYu], [LiuYu2], [LiuYu1]}%
, we assume the initial condition is compactly supported in the space
variable. This means that we want to understand the detailed propagation of
localized perturbation. Furthermore, we assume an extra exponential velocity
weight ($e^{\alpha |\xi |^{p}}$, $\alpha >0$ small and $0<p\leq 2$) on the
initial data. Under this assumption, we get an accurate relationship between
decay rates and weight functions.

The main idea of this paper is to combine the long wave-short wave
decomposition, the wave-remainder decomposition, the weighted energy
estimate and the regularization estimate together to analyze the solution.
The long-short wave decomposition, which is based on the Fourier transform,
gives the fluid structure or time decay estimate of the solution. The
wave-remainder decomposition is for extracting the singular waves. The
weighted energy estimate is for the pointwise estimate of the remainder
term, in which the regularization estimate is used. We explain the idea in
more details as below.

Inside the finite Mach number region, the solution is dominated by the long
wave part. In order to obtain its decay rate, we devise different methods
for $0\leq \gamma <1$ and $-2<\gamma <0$ respectively. For $0\leq \gamma <1$%
, taking advantage of the spectrum information of the Boltzmann collision
operator \cite{[Ellis]}, the Fourier multiplier techniques can be applied to
obtain the pointwise structure of the fluid part. However, for $-2<\gamma <0$%
, the spectrum information is missing due to the weak damping for large
velocity. Instead, we use similar arguments as those in the papers by
Kawashima \cite{[Kawashima]}, Strain \cite{[Strain]} and Strain-Guo \cite%
{[Strain-Guo]} to get optimal decay in time. It is shown that the $L^{2}$
norm of the short wave exponentially decays in time for $0\leq \gamma <1$
essentially due to the spectrum gap, while it decays only sub-exponentially
for $-2<\gamma <0$ if imposing an exponential velocity weight on the initial
data.

As mentioned before, we use the wave-remainder decomposition to extract the
singular waves in the short wave. This decomposition is based on a
Picard-type iteration. Such an iteration is manipulated to construct the
increasingly regular particlelike waves; in other words, the first several
terms in the iteration (indeed, the first seven terms of the iteration)
contain the most singular part of the solution, the so-called wave part. In
virtue of the pointwise estimate of the damped transport equation, we have a
rather accurate pointwise estimate for the wave part. On the other hand, the
regularization estimate enables us to show the remainder becomes regular,
and together with the $L^{2}$ decay of the short wave yields the $L^{\infty
} $ decay of the short wave. Combining this with the long wave, we establish
the pointwise structure inside the finite Mach number region.

As for the structure outside the finite Mach number region, it remains to
estimate the remainder part since we already have gained an explicit
estimate for the wave part. The weighted energy estimate plays an essential
role here. The weight functions not only are chosen delicately for different
$\gamma $\ and $p$, but also takes the domain decomposition into account. It
is noted that the sufficient understanding of the structure of the wave
part, which has been obtained previously, is needed in the estimate. And the
regularization estimate makes it possible to do the higher order weighted
energy estimate. Then the desired pointwise estimate follows from the
Sobolev inequality.

\subsection{Significant points of the paper}

We point out some significant points of this paper as follows:

\begin{itemize}
\item Singular waves: The wave-remainder decomposition plays an important
role in this paper. The significant points of the remainder part will be
discussed later. To comprehend the wave part (singular waves), we have to
establish a quantitative estimate of the damped transport operator $\mathbb{S%
}^{t}$ (see Lemma \ref{S^t-Sup}) first, since the singular waves can be
represented by the combination of operators $\mathbb{S}^{t}$ and $K$. If we
do not assume any velocity weight on the initial data, for the hard sphere
case (see \cite{[LiuYu], [LiuYu2], [LiuYu1]}), one can estimate the space
and time decay of the wave part precisely. However, for the hard potential
case (see \cite{[LeeLiuYu]}), the behavior seems not clear. In this regard,
we reinvestigate the hard potential case, as well as Maxwellian molecules
and soft potentials, assuming the initial condition is compactly supported
in $x$ and has a $L_{\xi }^{\infty }(e^{\alpha |\xi |^{p}})$ bound. Under
this assumption, we get exponential time decay for $0\leq \gamma <1$ and
sub-exponential time decay for $-2<\gamma <0$. Simultaneously, we get
sub-exponential space decay for $-2<\gamma <1$. This wave structure reveals
accurate dependences of decay rates on initial weights, as opposed to the
classical hard sphere case (\cite{[LiuYu]}). There are some interesting
observations: \newline
\newline
(a) For the soft potential case, we get sub-exponential time decay with rate
$e^{-\alpha ^{\frac{-\gamma }{p-\gamma }}t^{\frac{p}{p-\gamma }}}$, which
coincides with the results of Caflisch \cite{[Caflisch]} and Strain-Guo \cite%
{[Strain-Guo]}. These rates should be consistent since they studied the
torus case with zero moments, and our wave part excludes the fluid part of
the solution. \newline
\newline
(b) We give a very precise relation between initial velocity weights and the
asymptotic behavior of the solution ($|x|$ large), i.e., if the initial
condition is with weight $e^{\alpha |\xi |^{p}}$, then we have
sub-exponential decay $e^{-\alpha ^{\frac{1-\gamma }{p+1-\gamma }}\left\vert
x\right\vert ^{\frac{p}{p+1-\gamma }}}$. Moreover, the asymptotic behavior
of the wave part (Lemma \ref{S^t-Sup}) and the remainder part (Equation (\ref%
{remainder})) are matching in our estimate. \newline
\newline

\item Regularization estimate: The regularization estimate plays a crucial
role in this paper (see Lemma \ref{Regularity}), which enables us to obtain
the pointwise estimate without regularity assumption on the initial data. In
the proof of Lemma \ref{Regularity}, it reveals that the mixture of the two
operators $\mathbb{S}^{t}$ and $K$ transports the regularity in the
microscopic velocity $\xi $ from $K$ to the regularity in the space $x$.
Noticing that $K$ is an integral operator from $L_{\xi }^{2}$ to $H_{\xi
}^{1}$ only when $\gamma >-2$, that is why we confine ourselves to the case $%
\gamma >-2$ in this paper. This notion was firstly introduced by Liu and Yu
for the hard sphere case \cite{[LeeLiuYu], [LiuYu], [LiuYu2], [LiuYu1]} and
they call it as Mixture Lemma. However, their machinery is invalid for the
current situation, since it will result in the $e^{\infty \cdot \alpha
\left\vert \xi \right\vert ^{p}}$ weight imposed on the initial data. To
resolve this difficulty, we introduce the differential operator $\mathcal{D}%
_{t}=t\nabla _{x}+\nabla _{\xi }$, which commutes with the free transport
operator. This operator is crucial since it is a bridge between the $x$
derivative and the $\xi $ derivative. We remark that the crucial operator $\mathcal{D}_{t}$ was
firstly introduced in the paper by Gualdani, Mischler and Mouhot \cite%
{[Gualdani]}, and Wu \cite{[Wu]} applied it to reprove the Mixture Lemma
used in \cite{[LeeLiuYu], [LiuYu], [LiuYu2], [LiuYu1]}. Through this operator $\mathcal{D}_{t},
$ mixing the operator $\mathbb{S}^{t}$ with $K$ enough times will help the $%
\xi $ regularity transfer to the $x$ regularity (here "enough times" depends
on how many $\xi$ regularities we want to transfer) without any regularity
assumption in $\xi $ on the initial data. In other words, mixing operators $%
\mathbb{S}^{t}$ and $K$ enough times will lead to $x$ regularity \textbf{%
automatically. }\textbf{\ }
\newline

\item Weighted energy estimate: The pointwise estimate of the solution
outside the finite Mach number region is constructed by the weighted energy
estimate. The time dependent weight functions are chosen according to
different $\gamma $ (interactions between particles) and $p$ (initial
velocity weight). For the hard sphere case $\gamma =1$ (see Liu-Yu \cite%
{[LiuYu], [LiuYu2], [LiuYu1]}), the weight function depends only on the time
and the space variables, and exponentially grows in space (it takes the
form $\exp\{\frac{|x|-Mt}{D}\}$). Since it commutes
with the integral operator $K$, the estimate is relatively simple. However,
for this paper $-2<\gamma <1$, the weight function is much more complicated.
Indeed, it depends on the velocity variable as well and thus does not
commute with the integral operator $K$, which leads to the coercivity of
linearized collision operator cannot be applied directly and loss of control
of some terms at first glance. The difficulty is eventually overcome by fine
tuning the weight functions, introducing refined space-velocity domain
decomposition and analyzing the integral operator $K$ with weight
accordingly (see Lemma \ref{K-epsilon}).
\end{itemize}

The rest of this paper is organized as follows: We first prepare some basic
properties of the collision operator in section \ref{pre}. After that, we
construct the long wave-short wave decomposition in section \ref{LS} and the
wave-remainder decomposition in section \ref{WR}. Finally, we establish the
global wave structures in section \ref{Global}.

\section{Basic properties of the collision operator}

\label{pre}

For the linearized Boltzmann equation $\left( \ref{bot.1.d}\right) ,$ the
collision operator $L$ consists of a multiplicative operator $\nu (\xi )$
and an integral operator $K$:
\begin{equation*}
Lf=-\nu (\xi )f+Kf\,,
\end{equation*}%
where
\begin{equation*}
\nu (\xi )=\int B(\vartheta )|\xi -\xi _{\ast }|^{\gamma }\mathcal{M}(\xi
_{\ast })d\xi _{\ast }d\om\,,
\end{equation*}%
and $Kf=-K_{1}f+K_{2}f$ \ is defined as \cite{[Glassey], [Grad]}:
\begin{equation*}
K_{1}f=\int B(\vartheta )|\xi -\xi _{\ast }|^{\gamma }\mathcal{M}^{1/2}(\xi )%
\mathcal{M}^{1/2}(\xi _{\ast })f(\xi _{\ast })d\xi _{\ast }d\om\,,
\end{equation*}%
\begin{align*}
K_{2}f& =\int B(\vartheta )|\xi -\xi _{\ast }|^{\gamma }\mathcal{M}%
^{1/2}(\xi _{\ast })\mathcal{M}^{1/2}(\xi ^{\prime })f(\xi _{\ast }^{\prime
})d\xi _{\ast }d\om \\
& \quad +\int B(\vartheta )|\xi -\xi _{\ast }|^{\gamma }\mathcal{M}%
^{1/2}(\xi _{\ast })\mathcal{M}^{1/2}(\xi _{\ast }^{\prime })f(\xi ^{\prime
})d\xi _{\ast }d\om\,.
\end{align*}%
In this section we will present a number of properties and estimates of the
operators $L$, $\nu (\xi )$ and $K$. To begin with, we list some fundamental
properties of these operators, which can be found in \cite{[Caflisch],
[GolsePoupand], [Grad], [Strain-Guo]}.

\begin{lemma}
\label{basic} For any $g\in L_{\sigma }^{2}$, we have the coercivity
estimate of the linearized collision operator $L$:
\begin{equation*}
\left\langle g,Lg\right\rangle _{\xi }\lesssim -\left\vert \mathrm{P}%
_{1}g\right\vert _{L_{\sigma }^{2}}^{2}\,.
\end{equation*}%
For the multiplicative operator $\nu (\xi )$, there exist positive constants
$\nu _{0}$ and $\nu _{1}$ such that
\begin{equation}
\nu _{0}(1+\left\vert \xi \right\vert )^{\gamma }\leq \nu (\xi )\leq \nu
_{1}(1+\left\vert \xi \right\vert )^{\gamma }.  \label{nu-gamma}
\end{equation}%
Moreover, for each multi-index $\alpha $,
\begin{equation}
|\pa_{\xi }^{\alpha }\nu (\xi )|\lesssim \left\langle \xi \right\rangle
^{\gamma -|\alpha |}\,.  \label{bot.1.f}
\end{equation}%
For the integral operator $K$,
\begin{equation*}
Kf=-K_{1}f+K_{2}f=\int_{{\mathbb{R}}^{3}}-k_{1}(\xi ,\xi _{\ast })f(\xi
_{\ast })d\xi _{\ast }+\int_{{\mathbb{R}}^{3}}k_{2}(\xi ,\xi _{\ast })f(\xi
_{\ast })d\xi _{\ast }\,,
\end{equation*}%
the kernels $k_{1}(\xi ,\xi _{\ast })$ and $k_{2}(\xi ,\xi _{\ast })$
satisfy
\begin{equation*}
k_{1}(\xi ,\xi _{\ast })\lesssim |\xi -\xi _{\ast }|^{\gamma }\exp \left\{ -%
\frac{1}{4}\left( |\xi |^{2}+|\xi _{\ast }|^{2}\right) \right\} \,,
\end{equation*}%
and%
\begin{equation*}
k_{2}(\xi ,\xi _{\ast })=a\left( \xi ,\xi _{\ast },\kappa \right) \exp
\left( -\frac{(1-\kappa )}{8}\left[ \frac{\left( \left\vert \xi \right\vert
^{2}-\left\vert \xi _{\ast }\right\vert ^{2}\right) ^{2}}{\left\vert \xi
-\xi _{\ast }\right\vert ^{2}}+\left\vert \xi -\xi _{\ast }\right\vert ^{2}%
\right] \right) \,,
\end{equation*}%
for any $0<\kappa <1$, together with
\begin{equation*}
a(\xi ,\xi _{\ast },\kappa )\leq C_{\kappa }|\xi -\xi _{\ast }|^{-1}(1+|\xi
|+|\xi _{\ast }|)^{\gamma -1}\,.
\end{equation*}%
In addition, their derivatives as well have similar estimates, i.e.,
\begin{equation*}
|\nabla _{\xi }k_{1}(\xi ,\xi _{\ast })|,|\nabla _{\xi _{\ast }}k_{1}(\xi
,\xi _{\ast })|\lesssim |\xi -\xi _{\ast }|^{\gamma -1}\exp \left\{ -\frac{1%
}{4}\left( |\xi |^{2}+|\xi _{\ast }|^{2}\right) \right\} \,,
\end{equation*}%
\begin{equation*}
|\nabla _{\xi }k_{2}(\xi ,\xi _{\ast })|,|\nabla _{\xi _{\ast }}k_{2}(\xi
,\xi _{\ast })|\lesssim |\nabla _{\xi }a\left( \xi ,\xi _{\ast }\right)
|\exp \left( -\frac{(1-\kappa )}{8}\left[ \frac{\left( \left\vert \xi
\right\vert ^{2}-\left\vert \xi _{\ast }\right\vert ^{2}\right) ^{2}}{%
\left\vert \xi -\xi _{\ast }\right\vert ^{2}}+\left\vert \xi -\xi _{\ast
}\right\vert ^{2}\right] \right) \,.
\end{equation*}
\end{lemma}

According to the above estimates of the integral operator $K,$ it follows
that for any $g_{1},g_{2}\in L_{\sigma }^{2}\cap L_{\xi }^{2},$%
\begin{equation}
\left\vert \left\langle g_{1},Kg_{2}\right\rangle _{\xi }\right\vert
\lesssim \left\{
\begin{array}{l}
\displaystyle\left\vert g_{1}\right\vert _{L_{\xi }^{2}}\left\vert
g_{2}\right\vert _{L_{\xi }^{2}}\quad \hbox{for}\quad 0\leq \gamma <1\,, \\
\\
\displaystyle\left\vert g_{1}\right\vert _{L_{\sigma }^{2}}\left\vert
g_{2}\right\vert _{L_{\sigma }^{2}}\quad \hbox{for}\quad -2<\gamma <0\,,%
\end{array}%
\right.
\end{equation}%
and for any $g_{0}\in L^{2},$%
\begin{equation}
\Vert Kg_{0}\Vert _{H_{\xi }^{1}L_{x}^{2}}\lesssim \Vert g_{0}\Vert _{L^{2}}.
\end{equation}

Next, we will provide the sup norm estimate for the integral operator $K$,
which extends (6.2) in Proposition 6.1 of Caflisch \cite{[Caflisch]} to the
case $-2<\gamma <1.$

\begin{lemma}
\label{Exten. of Calfisch}Let $0<p\leq 2.$ For any $\beta _{1}\geq 0$ and $%
0\leq \beta _{2}<\frac{1}{4},$ the operator $K$ satisfies
\begin{equation}
\left\vert Kg\left( \xi \right) \right\vert =\left\vert \int k\left( \xi
,\xi _{\ast }\right) g\left( \xi _{\ast }\right) d\xi _{\ast }\right\vert
\lesssim \left\langle \xi \right\rangle ^{-\beta _{1}+\gamma -2}e^{-\beta
_{2}\left\vert \xi \right\vert ^{p}}\left\vert g\right\vert _{L_{\xi
}^{\infty }\left( \left\langle \xi \right\rangle ^{\beta _{1}}e^{\beta
_{2}\left\vert \xi \right\vert ^{p}}\right) }.  \label{K-1}
\end{equation}
\end{lemma}

\begin{proof}
We first give an estimate on the kernel $k,$ which extends Proposition 5.1
in \cite{[Caflisch]} to the case $-2<\gamma <1$. For any $0<\kappa <1,$ we
have
\begin{eqnarray*}
k_{1}\left( \xi ,\xi _{\ast }\right) &\lesssim &\left\vert \xi -\xi _{\ast
}\right\vert ^{\gamma }\exp \left\{ -\frac{1}{4}\left( \left\vert \xi
\right\vert ^{2}+\left\vert \xi _{\ast }\right\vert ^{2}\right) \right\} \\
&\leq &C_{\kappa }\left\vert \xi -\xi _{\ast }\right\vert ^{-2}\left(
1+\left\vert \xi \right\vert +\left\vert \xi _{\ast }\right\vert \right)
^{\gamma -1}\exp \left\{ -\frac{1}{4}\left( 1-\kappa \right) \left(
\left\vert \xi \right\vert ^{2}+\left\vert \xi _{\ast }\right\vert
^{2}\right) \right\} ,
\end{eqnarray*}%
for some constant $C_{\kappa }>0$. Since
\begin{equation*}
\frac{1}{4}\left( \left\vert \xi \right\vert ^{2}+\left\vert \xi _{\ast
}\right\vert ^{2}\right) \geq \frac{1}{8}\left( \left\vert \xi -\xi _{\ast
}\right\vert ^{2}+\frac{\left\vert \left\vert \xi \right\vert
^{2}-\left\vert \xi _{\ast }\right\vert ^{2}\right\vert ^{2}}{\left\vert \xi
-\xi _{\ast }\right\vert ^{2}}\right) ,
\end{equation*}%
we deduce
\begin{equation*}
k_{1}\left( \xi ,\xi _{\ast }\right) \lesssim \left\vert \xi -\xi _{\ast
}\right\vert ^{-2}\left( 1+\left\vert \xi \right\vert +\left\vert \xi _{\ast
}\right\vert \right) ^{\gamma -1}\exp \left\{ -\frac{\left( 1-\kappa \right)
}{8}\left( \left\vert \xi -\xi _{\ast }\right\vert ^{2}+\frac{\left\vert
\left\vert \xi \right\vert ^{2}-\left\vert \xi _{\ast }\right\vert
^{2}\right\vert ^{2}}{\left\vert \xi -\xi _{\ast }\right\vert ^{2}}\right)
\right\} .
\end{equation*}%
Together with $k_{2}\left( \xi ,\xi _{\ast }\right) ,$ it follows that for
any $0<\kappa <1,$%
\begin{equation}
\left\vert k\left( \xi ,\xi _{\ast }\right) \right\vert \lesssim \left\vert
\xi -\xi _{\ast }\right\vert ^{-2}\left( 1+\left\vert \xi \right\vert
+\left\vert \xi _{\ast }\right\vert \right) ^{\gamma -1}\exp \left\{ -\frac{%
\left( 1-\kappa \right) }{8}\left( \left\vert \xi -\xi _{\ast }\right\vert
^{2}+\frac{\left\vert \left\vert \xi \right\vert ^{2}-\left\vert \xi _{\ast
}\right\vert ^{2}\right\vert ^{2}}{\left\vert \xi -\xi _{\ast }\right\vert
^{2}}\right) \right\} \,.  \label{Est-for-k}
\end{equation}

Now, in view of $\left( \ref{Est-for-k}\right) ,$
\begin{eqnarray*}
&&\left\vert \int k\left( \xi ,\xi _{\ast }\right) g\left( \xi _{\ast
}\right) d\xi _{\ast }\right\vert \\
&\lesssim &\int_{\mathbb{R}^{3}}\frac{1}{\left\vert \xi -\xi _{\ast
}\right\vert ^{2}}\left( 1+\left\vert \xi \right\vert +\left\vert \xi _{\ast
}\right\vert \right) ^{\gamma -1}\exp \left( -\frac{\left( 1-\kappa \right)
}{8}\left( \left\vert \xi -\xi _{\ast }\right\vert ^{2}+\frac{\left\vert
\left\vert \xi \right\vert ^{2}-\left\vert \xi _{\ast }\right\vert
^{2}\right\vert ^{2}}{\left\vert \xi -\xi _{\ast }\right\vert ^{2}}\right)
\right) \left\vert g\left( \xi _{\ast }\right) \right\vert d\xi _{\ast } \\
&\lesssim &e^{-\beta _{2}\left\vert \xi \right\vert ^{p}}\left( 1+\left\vert
\xi \right\vert \right) ^{\gamma -1}\left\vert g\right\vert _{L_{\xi
}^{\infty }\left( \left\langle \xi \right\rangle ^{\beta _{1}}e^{\beta
_{2}\left\vert \xi \right\vert ^{p}}\right) }\cdot \mathbb{A}(\xi )\,,
\end{eqnarray*}%
where
\begin{equation*}
\mathbb{A}(\xi )=\int_{\mathbb{R}^{3}}\frac{1}{\left\vert \xi -\xi _{\ast
}\right\vert ^{2}}\exp \left\{ -\frac{\left( 1-\kappa \right) }{8}\left(
\left\vert \xi -\xi _{\ast }\right\vert ^{2}+\frac{\left\vert \left\vert \xi
\right\vert ^{2}-\left\vert \xi _{\ast }\right\vert ^{2}\right\vert ^{2}}{%
\left\vert \xi -\xi _{\ast }\right\vert ^{2}}\right) +\beta _{2}\left\vert
\xi \right\vert ^{p}-\beta _{2}\left\vert \xi _{\ast }\right\vert
^{p}\right\} \left( 1+\left\vert \xi _{\ast }\right\vert \right) ^{-\beta
_{1}}d\xi _{\ast }\,.
\end{equation*}%
Notice that
\begin{equation*}
\left\vert \xi -\xi _{\ast }\right\vert ^{2}+\frac{\left\vert \left\vert \xi
\right\vert ^{2}-\left\vert \xi _{\ast }\right\vert ^{2}\right\vert ^{2}}{%
\left\vert \xi -\xi _{\ast }\right\vert ^{2}}\geq 2\left\vert \left\vert \xi
\right\vert ^{2}-\left\vert \xi _{\ast }\right\vert ^{2}\right\vert ,
\end{equation*}%
and
\begin{equation*}
\left\vert \left\vert \xi \right\vert ^{p}-\left\vert \xi _{\ast
}\right\vert ^{p}\right\vert \leq \left\vert \left\vert \xi \right\vert
^{2}-\left\vert \xi _{\ast }\right\vert ^{2}\right\vert ^{\frac{p}{2}}.
\end{equation*}%
Picking $\kappa =\frac{1-4\beta _{2}}{2}$ and $\varpi =\frac{1-4\beta _{2}}{%
16}$ yields%
\begin{align*}
\mathbb{A}(\xi )& \lesssim \int_{\mathbb{R}^{3}}\frac{\left( 1+\left\vert
\xi _{\ast }\right\vert \right) ^{-\beta _{1}}}{\left\vert \xi -\xi _{\ast
}\right\vert ^{2}}\exp \left\{ -\varpi \left( \left\vert \xi -\xi _{\ast
}\right\vert ^{2}+\frac{\left\vert \left\vert \xi \right\vert
^{2}-\left\vert \xi _{\ast }\right\vert ^{2}\right\vert ^{2}}{\left\vert \xi
-\xi _{\ast }\right\vert ^{2}}\right) -\beta _{2}\left\vert \left\vert \xi
\right\vert ^{2}-\left\vert \xi _{\ast }\right\vert ^{2}\right\vert +\beta
_{2}\left\vert \left\vert \xi \right\vert ^{2}-\left\vert \xi _{\ast
}\right\vert ^{2}\right\vert ^{\frac{p}{2}}\right\} d\xi _{\ast } \\
& \lesssim \int_{\mathbb{R}^{3}}\frac{1}{\left\vert \xi -\xi _{\ast
}\right\vert ^{2}}\left( 1+\left\vert \xi _{\ast }\right\vert \right)
^{-\beta _{1}}\exp \left\{ -\varpi \left( \left\vert \xi -\xi _{\ast
}\right\vert ^{2}+\frac{\left\vert \left\vert \xi \right\vert
^{2}-\left\vert \xi _{\ast }\right\vert ^{2}\right\vert ^{2}}{\left\vert \xi
-\xi _{\ast }\right\vert ^{2}}\right) \right\} d\xi _{\ast } \\
& \equiv I\,,
\end{align*}%
since
\begin{equation*}
\exp \left( -\beta _{2}\left\vert \left\vert \xi \right\vert ^{2}-\left\vert
\xi _{\ast }\right\vert ^{2}\right\vert +\beta _{2}\left\vert \left\vert \xi
\right\vert ^{2}-\left\vert \xi _{\ast }\right\vert ^{2}\right\vert ^{\frac{p%
}{2}}\right) <e^{\beta _{2}},
\end{equation*}%
uniformly in $\xi ,\ \xi _{\ast }$ and $p.$ We split $I$ into two parts: $%
I_{1}$, with $\left\vert \xi _{\ast }\right\vert <\frac{1}{3}\left\vert \xi
\right\vert ,$ and $I_{2},$ with $\left\vert \xi _{\ast }\right\vert >\frac{%
1 }{3}\left\vert \xi \right\vert .$ Then%
\begin{equation}
I_{1}\lesssim e^{-\frac{\varpi }{4}\left\vert \xi \right\vert ^{2}}
\label{Integral I-1}
\end{equation}%
since $\left\vert \xi -\xi _{\ast }\right\vert ^{2}\geq \frac{4}{9}%
\left\vert \xi \right\vert ^{2}$ in that domain. In the domain integration
for $I_{2},$ we have $\left( 1+\left\vert \xi _{\ast }\right\vert \right) >%
\frac{1}{3}\left( 1+\left\vert \xi \right\vert \right) ,$ so that
\begin{eqnarray}
I_{2} &\lesssim &\left( 1+\left\vert \xi \right\vert \right) ^{-\beta
_{1}}\int_{\left\vert \xi _{\ast }\right\vert >\frac{1}{3}\left\vert \xi
\right\vert }\frac{1}{\left\vert \xi -\xi _{\ast }\right\vert ^{2}}\exp
\left\{ -\varpi \left( \left\vert \xi -\xi _{\ast }\right\vert ^{2}+\frac{%
\left\vert \left\vert \xi \right\vert ^{2}-\left\vert \xi _{\ast
}\right\vert ^{2}\right\vert ^{2}}{\left\vert \xi -\xi _{\ast }\right\vert
^{2}}\right) \right\} d\xi _{\ast }  \notag \\
&\lesssim &\left( 1+\left\vert \xi \right\vert \right) ^{-\beta _{1}-1},
\label{Integral I-2}
\end{eqnarray}%
due to Proposition 5.3 in \cite{[Caflisch]}. Combining $\left( \ref{Integral
I-1}\right) $ and $\left( \ref{Integral I-2}\right) ,$ we find
\begin{equation*}
I\lesssim \left( 1+\left\vert \xi \right\vert \right) ^{-\beta _{1}-1},
\end{equation*}%
and hence
\begin{equation*}
\left\vert Kg\left( \xi \right) \right\vert \lesssim \left\langle \xi
\right\rangle ^{-\beta _{1}+\gamma -2}e^{-\beta _{2}\left\vert \xi
\right\vert ^{p}}\left\vert g\right\vert _{L_{\xi }^{\infty }\left(
\left\langle \xi \right\rangle ^{\beta _{1}}e^{\beta _{2}\left\vert \xi
\right\vert ^{p}}\right) }.
\end{equation*}
\end{proof}

In fact, during the proof of this lemma, one can also infer that $k\left(
\xi ,\xi _{\ast }\right) $ is integrable in $\xi _{\ast }$ with%
\begin{equation}
\int_{\mathbb{R}^{3}}\left\vert k\left( \xi ,\xi _{\ast }\right) \right\vert
d\xi _{\ast }\lesssim \left( 1+\left\vert \xi \right\vert \right) ^{\gamma
-2},  \label{k-integ}
\end{equation}%
for $-2<\gamma <1.$

Regarding the weighted energy estimate, the following weight functions $\mu
(x,\xi )$ will be taken into account:
\begin{equation}
\mu (x,\xi )=1,\quad \hbox{or}\quad \exp \left( \epsilon \theta (x,\xi
)\right) \,,  \label{weight-function2}
\end{equation}%
where
\begin{align*}
\theta (x,\xi )& =5\Big(\de\left\langle x\right\rangle \Big)^{\frac{p}{p+1-%
\ga}}\left( 1-\chi \left( \de\left\langle x\right\rangle \left\langle \xi
\right\rangle ^{\ga-p-1}\right) \right) \\
& \quad +\bigg[\left( 1-\chi \Big(\de\left\langle x\right\rangle
\left\langle \xi \right\rangle ^{\ga-p-1}\Big)\right) \left[ \de\left\langle
x\right\rangle \right] \left\langle \xi \right\rangle ^{\ga-1}+3\left\langle
\xi \right\rangle ^{p}\bigg]\chi \left( \de\left\langle x\right\rangle
\left\langle \xi \right\rangle ^{\ga-p-1}\right) \,,
\end{align*}%
with $0<p\leq 2;$ the constants $\epsilon $ and$\ \delta >0$ will be chosen
sufficiently small later on. Among them, the choices of the functions $%
\theta $ are motivated by \cite{[cc]}. Under these considerations, we need
the estimates of $K$ as below.

\begin{lemma}
\label{K-esti-2}Let $0<p\leq 2$ and $g_{1},g_{2}\in L_{\sigma }^{2}\cap
L_{\xi }^{2}.\ $ Then for any $\epsilon \geq 0$ sufficiently small,%
\begin{equation}
\left\vert \left\langle g_{1},e^{\epsilon \theta \left( x,\xi \right)
}Ke^{-\epsilon \theta \left( x,\xi \right) }g_{2}\right\rangle _{\xi
}-\left\langle g_{1},Kg_{2}\right\rangle _{\xi }\right\vert \lesssim \left\{
\begin{array}{l}
\displaystyle\epsilon \left\vert g_{1}\right\vert _{L_{\xi }^{2}}\left\vert
g_{2}\right\vert _{L_{\xi }^{2}}\quad \hbox{for}\quad 0\leq \gamma <1\,, \\
\\
\displaystyle\epsilon \left\vert g_{1}\right\vert _{L_{\sigma
}^{2}}\left\vert g_{2}\right\vert _{L_{\sigma }^{2}}\quad \hbox{for}\quad
-2<\gamma <0.%
\end{array}%
\right.  \label{Weighted-K-1}
\end{equation}%
In particular,
\begin{equation}
\left\vert \left\langle g_{1},e^{\epsilon \theta \left( x,\xi \right)
}Ke^{-\epsilon \theta \left( x,\xi \right) }g_{2}\right\rangle _{\xi
}\right\vert \lesssim \left\{
\begin{array}{l}
\displaystyle\left\vert g_{1}\right\vert _{L_{\xi }^{2}}\left\vert
g_{2}\right\vert _{L_{\xi }^{2}}\quad \hbox{for}\quad 0\leq \gamma <1\,, \\
\\
\displaystyle\left\vert g_{1}\right\vert _{L_{\sigma }^{2}}\left\vert
g_{2}\right\vert _{L_{\sigma }^{2}}\quad \hbox{for}\quad -2<\gamma <0.%
\end{array}%
\right.  \label{Weighted-K-2}
\end{equation}%
Consequently, for $g_{0}\in L^{2}\left( \mu \right) ,$%
\begin{equation}
\Vert Kg_{0}\Vert _{L^{2}\left( \mu \right) }\lesssim \Vert g_{0}\Vert
_{L^{2}\left( \mu \right) }.  \label{K-3}
\end{equation}
\end{lemma}

\begin{proof}
It suffices to show that for $j=1$ and $2,$
\begin{equation}
\left\vert e^{\epsilon \theta \left( x,\xi \right) }K_{j}e^{-\epsilon \theta
\left( x,\xi \right) }-K_{j}\right\vert _{L_{\xi }^{2}}\lesssim \epsilon \,.
\label{Kj-weighted-estimate}
\end{equation}%
By the Cauchy-Schwartz inequality,
\begin{equation}
\frac{\left( \left\vert \xi \right\vert ^{2}-\left\vert \xi _{\ast
}\right\vert ^{2}\right) ^{2}}{\left\vert \xi -\xi _{\ast }\right\vert ^{2}}%
+\left\vert \xi -\xi _{\ast }\right\vert ^{2}\geq 2\left\vert \left\vert \xi
\right\vert ^{2}-\left\vert \xi _{\ast }\right\vert ^{2}\right\vert .
\label{Ineq1}
\end{equation}%
Further, rewrite
\begin{align*}
& e^{\epsilon \theta \left( x,\xi \right) }k_{2}\left( \xi ,\xi _{\ast
}\right) e^{-\epsilon \theta \left( x,\xi _{\ast }\right) }-k_{2}\left( \xi
,\xi _{\ast }\right) \\
& =\left\{ \widetilde{a}\left( \xi ,\xi _{\ast }\right) \exp \left( -\frac{1%
}{16}\left[ \frac{\left( \left\vert \xi \right\vert ^{2}-\left\vert \xi
_{\ast }\right\vert ^{2}\right) ^{2}}{\left\vert \xi -\xi _{\ast
}\right\vert ^{2}}+\left\vert \xi -\xi _{\ast }\right\vert ^{2}\right]
\right) \right\} \\
& \times \left\{ \exp \left( -\frac{1}{16}\left[ \frac{\left( \left\vert \xi
\right\vert ^{2}-\left\vert \xi _{\ast }\right\vert ^{2}\right) ^{2}}{%
\left\vert \xi -\xi _{\ast }\right\vert ^{2}}+\left\vert \xi -\xi _{\ast
}\right\vert ^{2}\right] \right) \times \left( \exp \left\{ \epsilon \left(
\theta \left( x,\xi \right) -\theta \left( x,\xi _{\ast }\right) \right)
\right\} -1\right) \right\} \\
& \equiv p\left( \xi ,\xi _{\ast }\right) s\left( \epsilon ,\xi ,\xi _{\ast
}\right) ,
\end{align*}%
where $\widetilde{a}\left( \xi ,\xi _{\ast }\right) =a\left( \xi ,\xi _{\ast
},\frac{1}{2}\right) $. We claim that
\begin{equation*}
\sup_{\xi ,\xi _{\ast }}\left\vert s\left( \epsilon ,\xi ,\xi _{\ast
}\right) \right\vert \rightarrow 0\ \ \text{as }\epsilon \rightarrow 0.
\end{equation*}%
Since $\left\vert \frac{\partial }{\partial \left\vert \xi \right\vert }%
\theta \left( x,\xi \right) \right\vert \lesssim \left\langle \xi
\right\rangle ^{p-2}\left\vert \xi \right\vert $ uniformly in $x$ for $\xi
\neq 0$ and $p\in (0,2],$ we obtain
\begin{equation}
\left\vert \theta \left( x,\xi \right) -\theta \left( x,\xi _{\ast }\right)
\right\vert =\left\vert \theta \left( x,\left\vert \xi \right\vert \right)
-\theta \left( x,\left\vert \xi _{\ast }\right\vert \right) \right\vert
\lesssim \left\langle \xi ^{\ast }\right\rangle ^{p-2}\left\vert \xi ^{\ast
}\right\vert \left\vert |\xi |-|\xi _{\ast }|\right\vert \leq
c_{1}\left\vert \left\vert \xi \right\vert ^{2}-\left\vert \xi _{\ast
}\right\vert ^{2}\right\vert ,  \label{theta-est}
\end{equation}%
for some $|\xi ^{\ast }|$ between $|\xi |$ and $|\xi _{\ast }|,$ and some
constant $c_{1}>0$ depending only upon $\gamma $ and $p$. Together with $%
\left( \ref{Ineq1}\right) $, whenever $\epsilon >0$ is sufficiently small
with $0\leq \epsilon c_{1}<\frac{1}{16},$
\begin{equation*}
\sup_{\xi ,\xi _{\ast }}\left\vert s\left( \epsilon ,\xi ,\xi _{\ast
}\right) \right\vert \leq \epsilon c_{1}\sup_{\xi ,\xi _{\ast }}\left(
\left\vert \left\vert \xi \right\vert ^{2}-\left\vert \xi _{\ast
}\right\vert ^{2}\right\vert \exp \left[ -\frac{1}{16}\left\vert \left\vert
\xi \right\vert ^{2}-\left\vert \xi _{\ast }\right\vert ^{2}\right\vert %
\right] \right) .
\end{equation*}%
In other words,
\begin{equation*}
\sup_{\xi ,\xi _{\ast }}\left\vert s\left( \epsilon ,\xi ,\xi _{\ast
}\right) \right\vert \rightarrow 0\ \ \text{as }\epsilon \rightarrow 0.
\end{equation*}%
Since $p\left( \xi ,\xi _{\ast }\right) $ is also a kernel of a bounded
operator on $L_{\xi }^{2}$ ($L_{\sigma }^{2}$) for $0\leq \gamma <1$ ($%
-2<\gamma <0$), this completes the estimate for $K_{2}$. As to the case $%
K_{1},$ it is easy and we omit the details.

According to the above discussion, we readily obtain that for $g_{0}\in
L^{2}\left( \mu\right) ,$%
\begin{equation*}
\Vert Kg_{0}\Vert_{L^{2}\left( \mu\right) }\lesssim\Vert g_{0}\Vert
_{L^{2}\left( \mu\right) }.
\end{equation*}
Precisely,

\begin{align*}
\Vert Kg_{0}\Vert_{L^{2}\left( \mu\right) } & =\sup_{g_{1}\in L^{2}\left(
\mu\right) ,\Vert g_{1}\Vert_{L^{2}(\mu)}\leq1}\int\left( Kg_{0}\right)
g_{1}\mu dxd\xi \\
& =\sup_{g_{1}\in L^{2}\left( \mu\right) ,\Vert
g_{1}\Vert_{L^{2}(\mu)}\leq1}\int\left\langle \mu^{1/2}K\mu^{-1/2}\left(
\mu^{1/2}g_{0}\right) ,\mu^{1/2}g_{1}\right\rangle _{L_{\xi}^{2}}dx \\
& \leq C\Vert\mu^{1/2}g_{0}\Vert_{L^{2}}=C\Vert g_{0}\Vert_{L^{2}\left(
\mu\right) }.
\end{align*}
\end{proof}

We here remark that this lemma also includes the following weighted
estimate: for any $g_{0}\in L^{2}\left( e^{\alpha |\xi |^{p}}\right) $ with $%
\alpha >0$ small and $0<p\leq 2$,
\begin{equation}
\Vert Kg_{0}\Vert _{L^{2}\left( e^{\alpha |\xi |^{p}}\right) }\lesssim \Vert
g_{0}\Vert _{L^{2}\left( e^{\alpha |\xi |^{p}}\right) }.  \label{K-Ex}
\end{equation}%
Before ending this section we recall the spectrum $\mathrm{{Spec}(\eta )}$, $%
\eta \in {\mathbb{R}}$, of the operator $-i\xi \cdot \eta +L$, in
preparation for estimating the Green's function of the linearized Boltzmann
equation in the next section.

\begin{lemma}
\cite{[Ellis]}\label{pr12} Set $\eta =|\eta |\om$. For any $0\leq \ga<1$,
there exist $\de>0$ and $\tau =\tau (\de)>0$ such that

\begin{enumerate}
\item For any $|\eta |>\de$,
\begin{equation*}
\hbox{\rm{Spec}}(\eta )\subset \{z\in \mathbb{C}:Re(z)<-\tau \}\,.
\end{equation*}

\item For any $|\eta |<\de$, the spectrum within the region $\{z\in \mathbb{C%
}:Re(z)>-\tau \}$ consists of exactly five eigenvalues $\{\varrho _{j}(\eta
)\}_{j=0}^{4}$,
\begin{equation*}
\hbox{\rm{Spec}}(\eta )\cap \{z\in \mathbb{C}:Re(z)>-\tau \}=\{\varrho
_{j}(\eta )\}_{j=0}^{4}\,,
\end{equation*}%
associated with corresponding eigenvectors $\{e_{j}(\eta )\}_{j=0}^{4}$.
They have the expansions
\begin{equation*}
\begin{array}{l}
\displaystyle\varrho _{j}(\eta )=-i\,a_{j}|\eta |-A_{j}|\eta |^{2}+O(|\eta
|^{3})\,, \\
\\
\displaystyle e_{j}(\eta )=E_{j}+O(|\eta |)\,,%
\end{array}%
\end{equation*}%
with $A_{j}>0$ and
\begin{equation*}
\left\{
\begin{array}{l}
a_{0}=\sqrt{\frac{5}{3}}\,,\quad a_{1}=-\sqrt{\frac{5}{3}}\,,\quad
a_{2}=a_{3}=a_{4}=0\,, \\[2mm]
E_{0}=\sqrt{\frac{3}{10}}\chi _{0}+\sqrt{\frac{1}{2}}\om\cdot \overline{\chi
}+\sqrt{\frac{1}{5}}\chi _{4}\,, \\[2mm]
E_{1}=\sqrt{\frac{3}{10}}\chi _{0}-\sqrt{\frac{1}{2}}\om\cdot \overline{\chi
}+\sqrt{\frac{1}{5}}\chi _{4}\,, \\[2mm]
E_{2}=-\sqrt{\frac{2}{5}}\chi _{0}+\sqrt{\frac{3}{5}}\chi _{4}\,, \\[2mm]
E_{3}=\om_{1}\cdot \overline{\chi }\,, \\[2mm]
E_{4}=\om_{2}\cdot \overline{\chi }\,,%
\end{array}%
\right.
\end{equation*}%
where $\overline{\chi }=(\chi _{1},\chi _{2},\chi _{3})$, and $\{\om_{1},\om%
_{2},\om\}$ is an orthonormal basis of ${\mathbb{R}}^{3}$. Here $%
\{e_{j}(\eta )\}_{j=0}^{4}$ can be normalized by $\big<e_{j}(-\eta
),e_{l}(\eta )\big>_{\xi }=\de_{jl}$, $0\leq j,l\leq 4.$

Moreover, the semigroup $e^{(-i\xi \cdot \eta +L)t}$ can be decomposed as
\begin{equation*}
\displaystyle e^{(-i\xi \cdot \eta +L)t}g=e^{(-i\xi \cdot \eta +L)t}\Pi
_{\eta }^{\perp }g+\mathbf{1}_{\{|\eta |<\de\}}\sum_{j=0}^{4}e^{\varrho
_{j}(\eta )t}\big<e_{j}(-\eta ),g\big>_{\xi }e_{j}(\eta )\,,
\end{equation*}%
where $\mathbf{1}_{\{\cdot \}}$ is the indicator function and there exists $%
C>0$ such that
\begin{equation*}
\left\vert e^{(-i\xi \cdot \eta +L)t}\Pi _{\eta }^{\perp }g\right\vert
_{L_{\xi }^{2}}\leq e^{-Ct}|g|_{L_{\xi }^{2}}\,.
\end{equation*}
\end{enumerate}
\end{lemma}

\section{Long wave-short wave decomposition}

\label{LS}

In order to study the large time behavior, we introduce the long wave-short
wave decomposition. By the Fourier transform, the solution of the linearized
Boltzmann equation can be written as
\begin{equation}
\displaystyle\mathbb{G}^{t}f_{0}=f(t,x,\xi )=\int_{{\mathbb{R}}^{3}}e^{i\eta
x+(-i\xi \cdot \eta +L)t}\widehat{f}_{0}(\eta ,\xi )d\eta \,,
\end{equation}%
where $\widehat{f}$ means the Fourier transform in the space variable and $%
\mathbb{G}^{t}$ is the solution operator (or Green's function) of the
linearized Boltzmann equation. We can decompose the solution $f$ into the
long wave part $f_{L}$ and the short wave part $f_{S}$ given respectively by
\begin{equation}
\begin{array}{l}
\label{bot.2.e}\displaystyle f_{L}=\int_{|\eta |<\de}e^{i\eta x+(-i\xi \cdot
\eta +L)t}\widehat{f}_{0}(\eta ,\xi )d\eta \,, \\
\\
\displaystyle f_{S}=\int_{|\eta |>\de}e^{i\eta x+(-i\xi \cdot \eta +L)t}%
\widehat{f}_{0}(\eta ,\xi )d\eta \,.%
\end{array}%
\end{equation}

For the case $0\leq \gamma <1$, we further decompose the long wave part as
the fluid part and non-fluid part, i.e., $f_{L}=f_{L;0}+f_{L;\perp }$, where
\begin{equation}
\begin{array}{l}
\label{bot.2.g}\displaystyle f_{L;0}=\int_{|\eta |<\de}\sum_{j=0}^{4}e^{%
\varrho _{j}(\eta )t}e^{i\eta x}\big<e_{j}(-\eta ),\hat{f_{0}}\big>_{\xi
}e_{j}(\eta )d\eta \,, \\
\\
\displaystyle f_{L;\perp }=\int_{|\eta |<\de}e^{i\eta x}e^{(-i\xi \cdot \eta
+L)t}\Pi _{\eta }^{\perp }\hat{f}_{0}d\eta \,.%
\end{array}%
\end{equation}%
Taking advantage of the spectrum information of the Boltzmann collision
operator (Lemma \ref{pr12}), we will obtain the $L^{2}$ estimates of the
non-fluid long wave part and short wave part directly. On the other hand,
the Fourier multiplier techniques can be applied to obtain the pointwise
structure of the fluid part. The estimate of this part is exactly the same
as in the Landau case \cite{[LWW1]} and hence we omit the details.

\begin{proposition}
\label{LS-estimate1}Let $0\leq \gamma <1$ and let $f$ be the solution of the
linearized Boltzmann equation.

\textrm{(a) (Fluid wave $f_{L;0}$)} Let $\mathbf{v}=\sqrt{5/3}$ be the sound
speed associated with the normalized global Maxwellian. For any given
positive integer $N$ and any given Mach number $\mathbb{M}>1$, there exists $%
C_{N}>0$ such that if $\left\vert x\right\vert \leq \left( \mathbb{M}%
+1\right) \mathbf{v}t$, then
\begin{align}
\left\vert f_{L;0}\right\vert _{L_{\xi }^{2}}& \leq C_{N}\left[ \left(
1+t\right) ^{-2}\left( 1+\frac{\left( \left\vert x\right\vert -\mathbf{v}%
t\right) ^{2}}{1+t}\right) ^{-N}+\left( 1+t\right) ^{-3/2}\left( 1+\frac{%
\left\vert x\right\vert ^{2}}{1+t}\right) ^{-N}\right. \\
& \quad \left. +\boldsymbol{1}_{\left\{ \left\vert x\right\vert \leq \mathbf{%
v}t\right\} }\left( 1+t\right) ^{-3/2}\left( 1+\frac{\left\vert x\right\vert
^{2}}{1+t}\right) ^{-3/2}\right] \left\Vert f_{0}\right\Vert
_{L_{x}^{1}L_{\xi }^{2}}.  \notag
\end{align}

\textrm{(b) (Non-fluid long wave $f_{L;\perp }$)} There exists a constant $%
c>0$ such that
\begin{equation}
\Vert f_{L;\perp }\Vert _{H_{x}^{s}L_{\xi }^{2}}\lesssim e^{-ct}\Vert
f_{0}\Vert _{L^{2}}  \label{bot.2.h}
\end{equation}%
for any $s>0$.

\textrm{(c) (Short wave $f_{S}$)} There exists a constant $c>0$ such that
\begin{equation}
\Vert f_{S}\Vert_{L^{2}}\lesssim e^{-ct}\Vert f_{0}\Vert_{L^{2}}\,.
\label{bot.2.f}
\end{equation}
\end{proposition}

Alternatively, for $-2<\gamma <0$, the spectrum information is missing due
to the weak damping for large velocity. Instead, we use similar arguments as
those in the papers by Kawashima \cite{[Kawashima]}, Strain \cite{[Strain]}
and Strain-Guo \cite{[Strain-Guo]} to get optimal time decay. All related
estimates have been done in \cite{[Strain-Guo]} and thereby we simply sketch
the proof.

\begin{proposition}
\label{LS-estimate2}Let $-2<\gamma <0$ and let $f$ be the solution of the
linearized Boltzmann equation. For $0<p\leq 2$ and $\alpha >0$ small, we have

\textrm{(a) (Long wave $f_{L}$)}
\begin{align}  \label{f_{L}soft}
\|f_{L}\|_{L^{\infty}_{x}L^{2}_{\xi}}\lesssim (1+t)^{-\frac{3}{2}}\left\Vert
f_{0}\right\Vert _{L_{x}^{1}L_{\xi}^{2}(e^{\alpha|\xi|^{p}})}\,.
\end{align}

\textrm{(b) (Short wave $f_{S}$)} There exists $c_{p,\gamma}>0$ such that
\begin{equation}  \label{f_{S}soft}
\left\Vert f_{S}\right\Vert _{L^{2}}\lesssim e^{-c_{p,\gamma}\alpha^{\frac{%
-\gamma}{p-\gamma}} t^{\frac{p}{p-\gamma}}}\left\Vert f_{0}\right\Vert
_{L^{2}(e^{\alpha|\xi|^{p}})}\,.
\end{equation}
\end{proposition}

\begin{proof}
Following the same argument as in \cite{[Strain]}, we find that there exists
a time-frequency functional $\mathcal{E}\left( t,\eta \right) $ such that
\begin{equation}
\mathcal{E}\left( t,\eta \right) \approx \left\vert \widehat{f}\left( t,\eta
\right) \right\vert _{L_{\xi }^{2}}^{2},
\end{equation}%
where for any $t>0$ and $\eta \in \mathbb{R}^{3}$, we have
\begin{equation}
\partial _{t}\mathcal{E}\left( t,\eta \right) +\sigma \widehat{\rho }\left(
\eta \right) \left\vert \widehat{f}\left( t,\eta \right) \right\vert
_{L_{\sigma }^{2}}^{2}\leq 0,  \label{E-N}
\end{equation}%
for some constant $\sigma >0.$ Here the notation $\widehat{\rho }\left( \eta
\right) =\min \{1,\left\vert \eta \right\vert ^{2}\}.$ Moreover, there
exists a weighted time-frequency functional $\mathcal{E}_{\alpha ,p}\left(
t,\eta \right) $ such that
\begin{equation}
\mathcal{E}_{\alpha ,p}\left( t,\eta \right) \approx \left\vert e^{\frac{%
\alpha |\xi |^{p}}{2}}\widehat{f}\left( t,\eta \right) \right\vert _{L_{\xi
}^{2}}^{2},
\end{equation}%
where for any $t>0$ and $\eta \in \mathbb{R}^{3}$, we have
\begin{equation}
\partial _{t}\mathcal{E}_{\alpha ,p}\left( t,\eta \right) \leq 0.  \label{EL}
\end{equation}

For the long wave part, the argument basically follows the paper \cite%
{[Strain]}. In fact, by (\ref{E-N}) and (\ref{EL}), we have
\begin{equation}
\Vert f_{L}\Vert _{H_{x}^{k}L_{\xi }^{2}}\lesssim (1+t)^{-\frac{3}{4}-\frac{k%
}{2}}\left\Vert f_{0}\right\Vert _{L_{x}^{1}L_{\xi }^{2}(e^{\alpha |\xi
|^{p}})}\,.  \label{fL}
\end{equation}%
With the aid of the Sobolev inequality, (\ref{f_{L}soft}) holds. For the
short wave $f_{S}$, applying the same argument as in Section 5 of \cite%
{[Strain-Guo]}, together with (\ref{E-N}) and (\ref{EL}), we get (\ref%
{f_{S}soft}).
\end{proof}

\section{Wave-remainder decomposition}

\label{WR}

In this section we introduce the wave-remainder decomposition, which is the
key decomposition in our paper. The strategy is to design a Picard-type
iteration, treating $Kf$ as a source term. Specifically, the zero order
approximation $h^{(0)}$ of the linearized Boltzmann equation (\ref{bot.1.d})
is defined as%
\begin{equation}
\left\{
\begin{array}{l}
\pa_{t}h^{(0)}+\xi \cdot \nabla _{x}h^{(0)}+\nu (\xi )h^{(0)}=0\,, \\[4mm]
h^{(0)}(0,x,\xi )=f_{0}(x,\xi )\,,%
\end{array}%
\right.  \label{bot.3.b}
\end{equation}%
and thus the difference $f-h^{(0)}$ satisfies
\begin{equation*}
\left\{
\begin{array}{l}
\pa_{t}(f-h^{(0)})+\xi \cdot \nabla _{x}(f-h^{(0)})+\nu (\xi
)(f-h^{(0)})=K(f-h^{(0)})+Kh^{(0)}\,, \\[4mm]
(f-h^{(0)})(0,x,\xi )=0\,.%
\end{array}%
\right.
\end{equation*}%
Therefore, the first order approximation $h^{(1)}$ can be defined as
\begin{equation}
\left\{
\begin{array}{l}
\pa_{t}h^{(1)}+\xi \cdot \nabla _{x}h^{(1)}+\nu (\xi )h^{(1)}=Kh^{(0)}\,, \\%
[4mm]
h^{(1)}(0,x,\xi )=0\,.%
\end{array}%
\right.  \label{bot.3.c}
\end{equation}%
In general, we can define the $j^{\mathrm{th}}$ order approximation $h^{(j)}$%
, $j\geq 1$, as
\begin{equation}
\left\{
\begin{array}{l}
\pa_{t}h^{(j)}+\xi \cdot \nabla _{x}h^{(j)}+\nu (\xi )h^{(j)}=Kh^{(j-1)}\,,
\\[4mm]
h^{(j)}(0,x,\xi )=0\,.%
\end{array}%
\right.  \label{bot.3.d}
\end{equation}%
Now, the wave part and the remainder part can be defined as follows:
\begin{equation}
W^{(6)}=\sum_{j=0}^{6}h^{(j)}\,,\quad \mathcal{R}^{(6)}=f-W^{(6)}\,,
\label{wave-remainder}
\end{equation}%
$\mathcal{R}^{(6)}$ solving the equation
\begin{equation}
\left\{
\begin{array}{l}
\pa_{t}\mathcal{R}^{(6)}+\xi \cdot \nabla _{x}\mathcal{R}^{(6)}=L\mathcal{R}%
^{(6)}+Kh^{(6)}\,, \\[4mm]
\mathcal{R}^{(6)}(0,x,\xi )=0\,.%
\end{array}%
\right.
\end{equation}%
In fact, $\mathcal{R}^{(6)}$ can be solved by using Green's function $%
\mathbb{G}^{t}$ for the full linearized Boltzmann equation, namely
\begin{equation}
\mathcal{R}^{(6)}=\int_{0}^{t}\mathbb{G}^{t-s}Kh^{(6)}(s)ds.  \label{R^(6)}
\end{equation}

\subsection{Estimates on the wave part}

We denote the solution operator of the damped transport equation
\begin{equation}
\left\{
\begin{array}{l}
\pa_{t}h+\xi \cdot \nabla _{x}h+\nu (\xi )h=0\,, \\[4mm]
h(0,x,\xi )=h_{0}\,,%
\end{array}%
\right.
\end{equation}%
by $\mathbb{S}^{t}$, i.e., $h(t)=\mathbb{S}^{t}h_{0}$. By method of
characteristics, the solution $\mathbb{S}^{t}h_{0}$ can be written down
explicitly; that is,
\begin{equation}
\mathbb{S}^{t}h_{0}(x,\xi )=h(t,x,\xi )=e^{-\nu (\xi )t}h_{0}(x-\xi t,\xi ).
\label{S^t}
\end{equation}%
In addition, it is easy to see that $h^{(j)}$ can be represented by the
combination of operators $\mathbb{S}^{t}$ and $K$.

In the sequel, we will find the pointwise decay of the solution $\mathbb{S}%
^{t}h_{0}$ in both time variable $t$ and space variable $x$ upon imposing
some weights on velocity variable $\xi $. Through the pointwise decay of the
solution $\mathbb{S}^{t}h_{0}$ and Duhamel's principle, we thereby obtain
the pointwise estimate of the wave part $W^{(6)}$. Moreover, we will provide
the $L^{2}$ estimate for $\mathbb{S}^{t}h_{0}$ with an exponential weight as
well, which leads us to obtain the $L^{2}$ estimates for $h^{(j)}$ $(0\leq
j\leq 6)$ and $\mathcal{R}^{(6)}.$

\begin{lemma}
\label{S^t-Sup} Let $\alpha >0$, $0<p\leq 2$ and $\beta>3/2$. Then for $%
0\leq \gamma <1$,
\begin{equation}
\left\vert \mathbb{S}^{t}h_{0}\right\vert _{L_{\xi }^{\infty
}\left(\left<\xi\right>^{\beta}\right)}\leq \sup_{y}e^{-c_{0}\left( t+\alpha
^{\frac{1-\gamma }{p+1-\gamma }}\left\vert x-y\right\vert ^{\frac{p}{%
p+1-\gamma }}\right) }\left\vert h_{0}\left( y,\cdot \right) \right\vert
_{L_{\xi }^{\infty }\left( e^{\alpha \left\vert \xi \right\vert
^{p}}\left<\xi\right>^{\beta}\right) }\,,  \label{S^t-Sup-H}
\end{equation}%
and for $-2<\gamma <0$,
\begin{equation}
\left\vert \mathbb{S}^{t}h_{0}\right\vert _{L_{\xi }^{\infty
}\left(\left<\xi\right>^{\beta}\right)}\leq C\left( \alpha ,\gamma \right)
\sup_{y}e^{-c_{0}\left( \alpha ^{\frac{-\gamma }{p-\gamma }}t^{\frac{p}{%
p-\gamma }}+\alpha ^{\frac{1-\gamma }{p+1-\gamma }}\left\vert x-y\right\vert
^{\frac{p}{p+1-\gamma }}\right) }\left\vert h_{0}\left( y,\cdot \right)
\right\vert _{L_{\xi }^{\infty }\left( e^{\alpha \left\vert \xi \right\vert
^{p}}\left<\xi\right>^{\beta}\right) }\,,  \label{S^t-Sup-S}
\end{equation}%
where $c_{0}=c\left( \gamma \right) >0$ and $C\left( \alpha ,\gamma \right)
>0$ are constants.
\end{lemma}

\begin{proof}
In view of $\left( \ref{S^t}\right) ,$ let $x-\xi t=y$ and then it suffices
to find the lower bound of
\begin{equation*}
\nu _{0}(t+\left\vert x-y\right\vert )^{\gamma }t^{1-\gamma }+\alpha
\left\vert x-y\right\vert ^{p}t^{-p}.
\end{equation*}%
Case 1. Hard potentials $0\leq \gamma <1$: \newline
Case 1a. As $p>\gamma .$ We discuss the lower bound separately in the two
regions
\begin{equation*}
|x-y|\leq \alpha ^{\frac{1}{\gamma -p}}t^{\frac{p+1-\gamma }{p-\gamma }}%
\text{ and }|x-y|>\alpha ^{\frac{1}{\gamma -p}}t^{\frac{p+1-\gamma }{%
p-\gamma }}.
\end{equation*}

If $|x-y|\leq \alpha ^{\frac{1}{\gamma -p}}t^{\frac{p+1-\gamma }{p-\gamma }%
}\left( \Leftrightarrow t\geq \alpha ^{\frac{1}{p+1-\gamma }}|x-y|^{\frac{%
p-\gamma }{p+1-\gamma }}\right) \,,$ then
\begin{equation*}
(t+\left\vert x-y\right\vert )^{\gamma }t^{1-\gamma }\geq
\begin{cases}
t\,,%
\vspace{3mm}
\\
|x-y|^{\gamma }t^{1-\gamma }\geq \alpha ^{\frac{1-\gamma }{p+1-\gamma }%
}|x-y|^{\frac{p}{p+1-\gamma }}\,,%
\end{cases}%
\end{equation*}%
which implies that
\begin{equation*}
(t+\left\vert x-y\right\vert )^{\gamma }t^{1-\gamma }\geq \frac{1}{2}\left(
t+\alpha ^{\frac{1-\gamma }{p+1-\gamma }}|x-y|^{\frac{p}{p+1-\gamma }%
}\right) .
\end{equation*}

If $|x-y|>\alpha^{\frac{1}{\gamma-p}}t^{\frac{p+1-\gamma}{p-\gamma}}\left(
\Leftrightarrow t<\alpha^{\frac{1}{p+1-\gamma}}|x-y|^{\frac{p-\gamma }{%
p+1-\gamma}}\,\right) ,$ then we have%
\begin{equation*}
(t+\left\vert x-y\right\vert )^{\gamma}t^{1-\gamma}\geq t\,,
\end{equation*}
and
\begin{equation*}
\alpha\left\vert x-y\right\vert ^{p}t^{-p}\geq\alpha^{\frac{1-\gamma }{%
p+1-\gamma}}|x-y|^{\frac{p}{p+1-\gamma}}\,.
\end{equation*}

As a consequence,
\begin{equation*}
\nu _{0}(t+\left\vert x-y\right\vert )^{\gamma }t^{1-\gamma }+\alpha
\left\vert x-y\right\vert ^{p}t^{-p}\geq c_{0}\left( t+\alpha ^{\frac{%
1-\gamma }{p+1-\gamma }}\left\vert x-y\right\vert ^{\frac{p}{p+1-\gamma }%
}\right) ,
\end{equation*}%
for some $c_{0}=c\left( \gamma \right) >0,$ so that
\begin{equation*}
\left\vert h(t,x,\cdot )\right\vert _{L_{\xi }^{\infty
}\left(\left<\xi\right>^{\beta}\right)}\leq \sup_{y}e^{-c\left( t+\alpha ^{%
\frac{1-\gamma }{p+1-\gamma }}\left\vert x-y\right\vert ^{\frac{p}{%
p+1-\gamma }}\right) }\left\vert h_{0}\left( y,\cdot \right) \right\vert
_{L_{\xi }^{\infty }\left( e^{\alpha \left\vert \xi \right\vert
^{p}}\left<\xi\right>^{\beta}\right) }\,.
\end{equation*}%
\newline
Case 1b. As $0<p<\gamma .$ We can apply a similar argument in Case 1a to
obtain $\left( \ref{S^t-Sup-H}\right) $ as well. \newline
Case 1c. As $0<p=\gamma ,$ it is easy to see that%
\begin{align*}
& \nu _{0}(t+\left\vert x-y\right\vert )^{\gamma }t^{1-\gamma }+\alpha
\left\vert x-y\right\vert ^{p}t^{-p} \\
& \geq \left( \nu _{0}\left\vert x-y\right\vert ^{\gamma }t^{1-\gamma
}\right) ^{\frac{p}{p}}+\left( \alpha \left\vert x-y\right\vert
^{p}t^{-p}\right) ^{\frac{1-\gamma }{1-\gamma }} \\
& \geq \nu _{0}^{\gamma }\alpha ^{1-\gamma }\left\vert x-y\right\vert ^{p},
\end{align*}%
due to Young's inequality. Therefore,
\begin{equation*}
\nu _{0}(t+\left\vert x-y\right\vert )^{\gamma }t^{1-\gamma }+\alpha
\left\vert x-y\right\vert ^{p}t^{-p}\geq
\begin{cases}
\nu _{0}t\,,%
\vspace{3mm}
\\
\nu _{0}^{\gamma }\alpha ^{1-\gamma }\left\vert x-y\right\vert ^{p},%
\end{cases}%
\end{equation*}%
which follows that
\begin{equation*}
\nu _{0}(t+\left\vert x-y\right\vert )^{\gamma }t^{1-\gamma }+\alpha
\left\vert x-y\right\vert ^{p}t^{-p}\geq c_{0}\left( t+\alpha ^{1-\gamma
}\left\vert x-y\right\vert ^{p}\right) ,
\end{equation*}%
for some $c_{0}=c\left( \gamma \right) >0,\ $as desired.\newline
Case 2. Soft potentials $-2<\gamma <0$: \newline
Case 2a. $\left\vert x-y\right\vert \leq \alpha ^{\frac{1}{\gamma -p}}t^{%
\frac{p+1-\gamma }{p-\gamma }}\left( \Leftrightarrow t\geq \alpha ^{\frac{1}{%
p+1-\gamma }}|x-y|^{\frac{p-\gamma }{p+1-\gamma }}\right) $ and $\left\vert
x-y\right\vert \geq t.$ In this case we have $t\geq \alpha $ and
\begin{equation*}
(t+\left\vert x-y\right\vert )^{\gamma }t^{1-\gamma }\geq \left\{
\begin{array}{l}
\left( t+\alpha ^{\frac{1}{\gamma -p}}t^{\frac{p+1-\gamma }{p-\gamma }%
}\right) ^{\gamma }t^{1-\gamma }\geq 2^{\gamma }\alpha ^{\frac{-\gamma }{%
p-\gamma }}t^{\frac{p}{p-\gamma }},\,%
\vspace{3mm}
\\
2^{\gamma }\left\vert x-y\right\vert ^{\gamma }t^{1-\gamma }\geq 2^{\gamma
}\alpha ^{\frac{1-\gamma }{p+1-\gamma }}|x-y|^{\frac{p}{p+1-\gamma }},%
\end{array}%
\right.
\end{equation*}%
so that
\begin{equation*}
(t+\left\vert x-y\right\vert )^{\gamma }t^{1-\gamma }\geq 2^{\gamma
-1}\left( \alpha ^{\frac{-\gamma }{p-\gamma }}t^{\frac{p}{p-\gamma }}+\alpha
^{\frac{1-\gamma }{p+1-\gamma }}|x-y|^{\frac{p}{p+1-\gamma }}\right) .
\end{equation*}%
Thus, $\left( \ref{S^t-Sup-S}\right) $ holds.\qquad \qquad \qquad \qquad
\qquad \qquad \qquad \qquad \qquad \qquad \qquad \qquad\ \newline
Case 2b. $\left\vert x-y\right\vert \leq \alpha ^{\frac{1}{\gamma -p}}t^{%
\frac{p+1-\gamma }{p-\gamma }}\left( \Leftrightarrow t\geq \alpha ^{\frac{1}{%
p+1-\gamma }}|x-y|^{\frac{p-\gamma }{p+1-\gamma }}\right) $ and $\left\vert
x-y\right\vert \leq t.$ In this case we have%
\begin{equation*}
\left\vert x-y\right\vert \leq \min \{\alpha ^{\frac{1}{\gamma -p}}t^{\frac{%
p+1-\gamma }{p-\gamma }},t\}.
\end{equation*}

If $\alpha ^{\frac{1}{\gamma -p}}t^{\frac{p+1-\gamma }{p-\gamma }}\geq t,$
then $t\geq \alpha $ and%
\begin{equation*}
(t+\left\vert x-y\right\vert )^{\gamma }t^{1-\gamma }\geq \left\{
\begin{array}{l}
\left( t+\alpha ^{\frac{1}{\gamma -p}}t^{\frac{p+1-\gamma }{p-\gamma }%
}\right) ^{\gamma }t^{1-\gamma }\geq 2^{\gamma }\alpha ^{\frac{-\gamma }{%
p-\gamma }}t^{\frac{p}{p-\gamma },}%
\vspace{3mm}
\\
2^{\gamma }t=2^{\gamma }t^{\frac{1-\gamma }{p+1-\gamma }}t^{\frac{p}{%
p+1-\gamma }}\geq 2^{\gamma }\alpha ^{\frac{1-\gamma }{p+1-\gamma }}|x-y|^{%
\frac{p}{p+1-\gamma }},%
\end{array}%
\right.
\end{equation*}%
which implies that
\begin{equation*}
(t+\left\vert x-y\right\vert )^{\gamma }t^{1-\gamma }\geq 2^{\gamma
-1}\left( \alpha ^{\frac{-\gamma }{p-\gamma }}t^{\frac{p}{p-\gamma }}+\alpha
^{\frac{1-\gamma }{p+1-\gamma }}|x-y|^{\frac{p}{p+1-\gamma }}\right) .
\end{equation*}%
Hence, $\left( \ref{S^t-Sup-S}\right) $ holds.

If $\alpha ^{\frac{1}{\gamma -p}}t^{\frac{p+1-\gamma }{p-\gamma }}\leq t,$
then we deduce $t\leq \alpha $ and thus $\left\vert x-y\right\vert \leq
\alpha .$ Since
\begin{equation*}
(t+\left\vert x-y\right\vert )^{\gamma }t^{1-\gamma }\geq 2^{\gamma }t\geq
2^{\gamma }\alpha ^{\frac{1}{p+1-\gamma }}|x-y|^{\frac{p-\gamma }{p+1-\gamma
}},
\end{equation*}%
we have
\begin{equation*}
(t+\left\vert x-y\right\vert )^{\gamma }t^{1-\gamma }\geq 2^{\gamma
-1}\left( t+\alpha ^{\frac{1}{p+1-\gamma }}|x-y|^{\frac{p-\gamma }{%
p+1-\gamma }}\right) .
\end{equation*}%
Together with the fact that $t\leq \alpha $ and $\left\vert x-y\right\vert
\leq \alpha ,$ we deduce
\begin{align*}
\left| h(t,x,\xi )\right|& \leq e^{-\nu _{0}(t+\left\vert x-y\right\vert
)^{\gamma }t^{1-\gamma }}\left|h_{0}(y,\xi)\right| \\
& \leq e^{-2^{\gamma -1}\nu _{0}t}\cdot e^{-2^{\gamma -1}\nu _{0}\alpha ^{%
\frac{1}{p+1-\gamma }}|x-y|^{\frac{p-\gamma }{p+1-\gamma }}}
\left|h_{0}(y,\xi)\right| \\
& \leq \left[ C_{1}e^{-2^{\gamma -1}\nu _{0}\alpha ^{\frac{-\gamma }{%
p-\gamma }}t^{\frac{p}{p-\gamma }}}\right] \left[ C_{1}e^{-2^{\gamma -1}\nu
_{0}\alpha ^{\frac{1-\gamma }{p+1-\gamma }}|x-y|^{\frac{p}{p+1-\gamma }}}%
\right]\left|h_{0}(y,\xi)\right| \\
& =Ce^{-2^{\gamma -1}\nu _{0}\left( \alpha ^{\frac{-\gamma }{p-\gamma }}t^{%
\frac{p}{p-\gamma }}+\alpha ^{\frac{1-\gamma }{p+1-\gamma }}|x-y|^{\frac{p}{%
p+1-\gamma }}\right) }\left|h_{0}(y,\xi)\right|,\ \ \ C=C_{1}^{2},
\end{align*}%
where
\begin{equation*}
\exp \left( 2^{\gamma -1}\nu _{0}\left( \alpha ^{\frac{-\gamma }{p-\gamma }%
}t^{\frac{p}{p-\gamma }}-t\right) \right) \leq \exp \left[ 2^{\gamma }\nu
_{0}\alpha \right] =C_{1}\left( \alpha ,\gamma \right) ,
\end{equation*}%
and%
\begin{equation*}
\exp \left[ 2^{\gamma -1}\nu _{0}\left( \alpha ^{\frac{1-\gamma }{p+1-\gamma
}}|x-y|^{\frac{p}{p+1-\gamma }}-\alpha ^{\frac{1}{p+1-\gamma }}|x-y|^{\frac{%
p-\gamma }{p+1-\gamma }}\right) \right] \leq \exp \left( 2^{\gamma }\nu
_{0}\alpha \right) =C_{1}\left( \alpha ,\gamma \right) .
\end{equation*}%
Thus, $\left( \ref{S^t-Sup-S}\right) $ holds.\qquad \qquad \qquad \qquad
\qquad \qquad \qquad \qquad \qquad \qquad \qquad \qquad\ \newline
Case 2c. $|x-y|>\alpha ^{\frac{1}{\gamma -p}}t^{\frac{p+1-\gamma }{p-\gamma }%
}\left( \Leftrightarrow t<\alpha ^{\frac{1}{p+1-\gamma }}|x-y|^{\frac{%
p-\gamma }{p+1-\gamma }}\right) \,.$ In this case we have%
\begin{equation*}
\alpha \left\vert x-y\right\vert ^{p}t^{-p}\geq
\begin{cases}
\alpha ^{\frac{-\gamma }{p-\gamma }}t^{\frac{p}{p-\gamma }}\,, \\
\alpha ^{\frac{1-\gamma }{p+1-\gamma }}|x-y|^{\frac{p}{p+1-\gamma }}\,,%
\end{cases}%
\end{equation*}%
so that
\begin{equation*}
\left\vert h(t,x,\cdot )\right\vert _{L_{\xi }^{\infty }\left( \left\langle
\xi \right\rangle ^{\beta }\right) }\leq \sup_{y}e^{-\frac{1}{2}\left(
\alpha ^{\frac{-\gamma }{p-\gamma }}t^{\frac{p}{p-\gamma }}+\alpha ^{\frac{%
1-\gamma }{p+1-\gamma }}\left\vert x-y\right\vert ^{\frac{p}{p+1-\gamma }%
}\right) }\left\vert h_{0}\left( y,\cdot \right) \right\vert _{L_{\xi
}^{\infty }\left( e^{\alpha \left\vert \xi \right\vert ^{p}}\left\langle \xi
\right\rangle ^{\beta }\right) }\,.
\end{equation*}
\end{proof}

Immediately from Lemmas \ref{Exten. of Calfisch} and \ref{S^t-Sup}, we get
the pointwise estimate of $h^{(j)},$ $0\leq j\leq 6,$ as below.

\begin{lemma}[Pointwise estimate of $h^{(j)}$, $0\leq j\leq 6$]
\label{pointwise-h} Let $f_{0}\left( x,\cdot \right) \in L_{\xi }^{\infty
}\left( e^{7\alpha \left\vert \xi \right\vert
^{p}}\left<\xi\right>^{\beta}\right) $ with compact support in variable $x$,
where $0<p\leq 2$, $\beta>3/2$ and $\alpha >0$ is small. Then there exists $%
c_{0}=c\left( \gamma \right) >0$ such that for $0\leq \gamma <1$,
\begin{equation*}
\left\vert h^{(j)}\right\vert _{L_{\xi }^{\infty
}\left(\left<\xi\right>^{\beta}\right)}\lesssim t^{j}e^{-c_{0}\left(
t+\alpha ^{\frac{1-\gamma }{p+1-\gamma }}\left\vert x\right\vert ^{\frac{p}{%
p+1-\gamma }}\right) }\left\Vert f_{0}\right\Vert _{L_{x}^{\infty }L_{\xi
}^{\infty }\left( e^{(j+1)\alpha \left\vert \xi \right\vert
^{p}}\left<\xi\right>^{\beta}\right) }\,,
\end{equation*}%
and for $-2<\gamma <0$,
\begin{equation*}
\left\vert h^{(j)}\right\vert _{L_{\xi }^{\infty
}\left(\left<\xi\right>^{\beta}\right)}\lesssim t^{j}e^{-c_{0}\left( \alpha
^{\frac{-\gamma }{p-\gamma }}t^{\frac{p}{p-\gamma }}+\alpha ^{\frac{1-\gamma
}{p+1-\gamma }}\left\vert x\right\vert ^{\frac{p}{p+1-\gamma }}\right)
}\left\Vert f_{0}\right\Vert _{L_{x}^{\infty }L_{\xi }^{\infty }\left(
e^{(j+1)\alpha \left\vert \xi \right\vert
^{p}}\left<\xi\right>^{\beta}\right) }\,.
\end{equation*}
\end{lemma}

In order to get the $L^{2}$ estimate of $h^{(j)}$, we need the $L^{2}$
estimate of the damped transport operator $\mathbb{S}^{t}$.

\begin{lemma}
\label{S^t-L2} Let $\alpha >0$ and $0<p\leq 2.$ Then for $0\leq \gamma <1,$%
\begin{equation*}
\left\Vert \mathbb{S}^{t}h_{0}\right\Vert _{L^{2}}\lesssim e^{-\nu
_{0}t}\left\Vert h_{0}\right\Vert _{L^{2}},
\end{equation*}%
and for $-2<\gamma <0,$%
\begin{equation*}
\left\Vert \mathbb{S}^{t}h_{0}\right\Vert _{L^{2}}\lesssim e^{-c\alpha ^{%
\frac{-\gamma }{p-\gamma }}t^{\frac{p}{p-\gamma }}}\left\Vert
h_{0}\right\Vert _{L^{2}\left( e^{\alpha \left\vert \xi \right\vert
^{p}}\right) },
\end{equation*}%
where the constant $c>0$ depends only upon $\gamma \ $and $p.$
\end{lemma}

\begin{proof}
We only prove the case $-2< \gamma<0$ since the case $0\leq \gamma <1$ is
obvious. In view of $\left( \ref{nu-gamma}\right) $ and $\left( \ref{S^t}%
\right) ,$
\begin{equation*}
\left\Vert \mathbb{S}^{t}h_{0}\right\Vert _{L^{2}}\leq \left( \sup_{\xi
}e^{-\nu _{0}t\left( 1+\left\vert \xi \right\vert \right) ^{\gamma }-\alpha
\left\vert \xi \right\vert ^{p}}\right) \left\Vert h_{0}\right\Vert
_{L^{2}\left( e^{\alpha \left\vert \xi \right\vert ^{p}}\right) },
\end{equation*}%
since $h_{0}$ has compact support in $x.$ As for $-2<\gamma <0,$%
\begin{eqnarray*}
\left\Vert \mathbb{S}^{t}h_{0}\right\Vert _{L^{2}} &\leq &\left( \sup_{\xi
}e^{-\nu _{0}t\left( 1+\left\vert \xi \right\vert \right) ^{\gamma }-\alpha
\left\vert \xi \right\vert ^{p}}\right) \left\Vert h_{0}\right\Vert
_{L^{2}\left( e^{\alpha \left\vert \xi \right\vert ^{p}}\right) } \\
&\leq &\left( \sup_{\left\vert \xi \right\vert \leq 1}e^{-\nu _{0}t\left(
1+\left\vert \xi \right\vert \right) ^{\gamma }-\alpha \left\vert \xi
\right\vert ^{p}}+\sup_{\left\vert \xi \right\vert >1}e^{-\nu _{0}t\left(
1+\left\vert \xi \right\vert \right) ^{\gamma }-\alpha \left\vert \xi
\right\vert ^{p}}\right) \left\Vert h_{0}\right\Vert _{L^{2}\left( e^{\alpha
\left\vert \xi \right\vert ^{p}}\right) } \\
&\leq &\left( e^{-2^{\gamma }\nu _{0}t}+\sup_{\left\vert \xi \right\vert
>1}e^{-2^{\gamma }\nu _{0}t\left\vert \xi \right\vert ^{\gamma }-\alpha
\left\vert \xi \right\vert ^{p}}\right) \left\Vert h_{0}\right\Vert
_{L^{2}\left( e^{\alpha \left\vert \xi \right\vert ^{p}}\right) } \\
&\leq &\left( e^{-2^{\gamma }\nu _{0}t}+e^{-\left( \frac{p-\gamma }{-\gamma }%
\left( \frac{-\gamma 2^{\gamma }\nu _{0}}{p}\right) ^{\frac{p}{p-\gamma }%
}\alpha ^{\frac{-\gamma }{p-\gamma }}\right) t^{\frac{p}{p-\gamma }}}\right)
\left\Vert h_{0}\right\Vert _{L^{2}\left( e^{\alpha \left\vert \xi
\right\vert ^{p}}\right) } \\
&\lesssim &e^{-c\alpha ^{\frac{-\gamma }{p-\gamma }}t^{\frac{p}{p-\gamma }%
}}\left\Vert h_{0}\right\Vert _{L^{2}\left( e^{\alpha \left\vert \xi
\right\vert ^{p}}\right) }
\end{eqnarray*}%
for some constant $c=c\left( \gamma ,p\right) >0,$ since $2^{\gamma }\nu
_{0}t\left\vert \xi \right\vert ^{\gamma }+\alpha \left\vert \xi \right\vert
^{p}$ attains a minimum at $\left\vert \xi \right\vert =\left( \frac{-\gamma
2^{\gamma }\nu _{0}t}{\alpha p}\right) ^{\frac{1}{p-\gamma }}.$
\end{proof}

Combining Lemma \ref{S^t-L2} and (\ref{K-Ex}), we thereby get the $L^{2}$
estimates of $h^{(j)}$, $0\leq j\leq 6$.

\begin{lemma}[$L^{2}$ estimate of $h^{(j)},$ $0\leq j\leq 6$]
\label{L2-h} Let $f_{0}\left( x,\cdot \right) \in L^{2}\left( e^{7\alpha
\left\vert \xi \right\vert ^{p}}\right) ,$ where $0<p\leq 2$ and $\alpha >0$
is small. Then there exists a constant $c>0$ such that for $0\leq \gamma <1$%
,
\begin{equation*}
\left\Vert h^{(j)}\right\Vert _{L^{2}}\lesssim t^{j}e^{-\nu _{0}t}\left\Vert
f_{0}\right\Vert _{L^{2}}\,,
\end{equation*}%
and for $-2<\gamma <0$,
\begin{equation*}
\left\Vert h^{(j)}\right\Vert _{L^{2}}\lesssim t^{j}e^{-c\alpha ^{\frac{%
-\gamma }{p-\gamma }}t^{\frac{p}{p-\gamma }}}\left\Vert f_{0}\right\Vert
_{L^{2}\left( e^{(j+1)\alpha \left\vert \xi \right\vert ^{p}}\right) }\,.
\end{equation*}
\end{lemma}

\subsection{Regularization estimate}

In the previous subsection, we have carried out the pointwise estimate of
the wave part and the $L^{2}$ estimates of $\mathcal{R}^{(6)}$. To obtain
the pointwise estimate on $\mathcal{R}^{(6)},$ we still need the
regularization estimate for $\mathcal{R}^{(6)}.$ In light of (\ref{R^(6)}),
we turn to the regularization estimate for $h^{(6)}$ in advance. To proceed,
we introduce a differential operator:
\begin{equation*}
\mathcal{D}_{t}=t\nabla _{x}+\nabla _{\xi }\,.
\end{equation*}%
This operator $\mathcal{D}_{t}$ is important since it commutes with the free
transport operator, i.e.,
\begin{equation*}
\lbrack \mathcal{D}_{t},\partial _{t}+\xi \cdot \nabla _{x}]=0\,,
\end{equation*}%
where $[A,B]=AB-BA$ is the commutator.

\begin{lemma}
\label{Operators-KSD} For any $g_{0}\in L_{x}^{2}H_{\xi }^{1}(\mu )$, we
have
\begin{equation}
\Vert Kg_{0}\Vert _{L_{x}^{2}H_{\xi }^{1}(\mu )}\lesssim \Vert g_{0}\Vert
_{L^{2}(\mu )}\,,\ \Vert K(\nabla _{\xi }g)\Vert _{L^{2}(\mu )}\lesssim
\Vert g_{0}\Vert _{L^{2}(\mu )}\,,  \label{K-4}
\end{equation}%
\begin{equation}
\Vert \mathbb{S}^{t}g_{0}\Vert _{L^{2}(\mu )}\lesssim e^{-c_{\gamma }t}\Vert
g_{0}\Vert _{L^{2}(\mu )}\,,  \label{1}
\end{equation}%
\begin{equation}
\displaystyle\Vert \mathcal{D}_{t}\mathbb{S}^{t}g_{0}\Vert _{L^{2}(\mu
)}\lesssim \left( 1+t\right) e^{-c_{\gamma }t}\Vert g_{0}\Vert
_{L_{x}^{2}H_{\xi }^{1}(\mu )}\,,  \label{ml.1}
\end{equation}%
here $c_{\gamma }>0$ for $0\leq \gamma <1$ and $c_{\gamma }=0$ for $%
-2<\gamma <0$.

Consequently,%
\begin{equation}
\Vert \mathcal{D}_{t}\mathbb{S}^{t}Kg_{0}\Vert _{L^{2}(\mu )}\lesssim \left(
1+t\right) e^{-c_{\gamma }t}\Vert g_{0}\Vert _{L^{2}(\mu )}.  \label{Dt-S-K}
\end{equation}
\end{lemma}

\begin{proof}
The estimate of (\ref{K-4}) can follow the same procedure as in Lemma \ref%
{K-esti-2} and hence we omit the details.

Denote $g\left( t\right) =\mathbb{S}^{t}g_{0}.$ Direct Computation shows
that
\begin{align*}
\frac{1}{2}\frac{d}{dt}\Vert g\Vert _{L^{2}(\mu )}^{2}& =\int \left( -\xi
\cdot \nabla _{x}g-\nu \left( \xi \right) g\right) g\mu dxd\xi =\frac{1}{2}%
\int \left( \xi \cdot \nabla _{x}\mu \right) \left\vert g\right\vert
^{2}dxd\xi -\int \nu \left( \xi \right) \left\vert g\right\vert ^{2}\mu
dxd\xi \\
& \leq \epsilon \delta c\int \left\langle \xi \right\rangle ^{\gamma
}\left\vert g\right\vert ^{2}\mu dxd\xi -\int \nu \left( \xi \right)
\left\vert g\right\vert ^{2}\mu dxd\xi ,
\end{align*}%
since $\left\vert \nabla _{x}\mu \right\vert \lesssim \epsilon \delta
\left\langle \xi \right\rangle ^{\gamma -1}$. After choosing $\epsilon ,\
\delta >0$ sufficiently small with $c\epsilon \delta <\nu _{0}$, we have%
\begin{equation*}
\frac{1}{2}\frac{d}{dt}\Vert g\Vert _{L^{2}(\mu )}^{2}\leq -c^{\prime }\int
\left\langle \xi \right\rangle ^{\gamma }\left\vert g\right\vert ^{2}\mu
dxd\xi =-c^{\prime }\Vert g_{0}\Vert _{L_{\sigma }^{2}(\mu )}^{2},
\end{equation*}%
for some constant $c^{\prime }=c\left( \gamma \right) >0.$ As a result,
\begin{equation}
\Vert g\Vert _{L^{2}(\mu )}\leq e^{-c_{\gamma }t}\Vert g_{0}\Vert
_{L^{2}(\mu )},  \label{2}
\end{equation}%
here the constant $c_{\gamma }>0$ for $0\leq \gamma <1$ and $c_{\gamma }=0$
for $-2<\gamma <0$.

Furthermore, set $y=\mathcal{D}_{t}\mathbb{S}^{t}g_{0}=\mathcal{D}_{t}g$ and
then $y$ satisfies the equation
\begin{equation*}
\left\{
\begin{array}{l}
\partial _{t}y+\xi \cdot \nabla _{x}y=-\nu \left( \xi \right) y-\left(
\nabla _{\xi }\nu \left( \xi \right) \right) g,%
\vspace{3mm}
\\
y\left( 0,x,\xi \right) =\left( \nabla _{\xi }g_{0}\right) \left( x,\xi
\right) .%
\end{array}%
\right.
\end{equation*}%
Immediately, by Duhamel's principle and $\left( \ref{2}\right) $
\begin{align*}
\left\Vert \mathcal{D}_{t}\mathbb{S}^{t}g_{0}\right\Vert _{L^{2}\left( \mu
\right) }& =\left\Vert y\right\Vert _{L^{2}\left( \mu \right) }\lesssim
e^{-c_{\gamma }t}\Vert \nabla _{\xi }g_{0}\Vert _{L^{2}(\mu
)}+\int_{0}^{t}e^{-c_{\gamma }\left( t-s\right) }\Vert g\left( s\right)
\Vert _{L^{2}(\mu )}ds \\
& \lesssim e^{-c_{\gamma }t}\left( \Vert \nabla _{\xi }g_{0}\Vert
_{L^{2}(\mu )}+t\Vert g_{0}\Vert _{L^{2}(\mu )}\right) \lesssim \left(
1+t\right) e^{-c_{\gamma }t}\Vert g_{0}\Vert _{L_{x}^{2}H_{\xi }^{1}(\mu )}.
\end{align*}
\end{proof}

We are now in the position to get the regularization estimate of $h^{(6)}$.
We find that without any regularity assumption on the initial condition, $%
h^{(6)}$ has $H_{x}^{2}$ regularity automatically.

\begin{lemma}[Regularization estimate on $h^{(6)}$]
\label{Regularity}
\begin{equation*}
\Vert h^{(6)}\Vert _{H_{x}^{2}L_{\xi }^{2}\left( \mu \right) }\lesssim
t^{4}\left( 1+t\right) ^{2}e^{-c_{\gamma }t}\Vert f_{0}\Vert _{L^{2}(\mu )},
\end{equation*}%
here $c_{\gamma }>0$ for $0\leq \gamma <1$ and $c_{\gamma }=0$ for $%
-2<\gamma <0$.
\end{lemma}

\begin{proof}
It follows immediately from Lemma \ref{Operators-KSD} that%
\begin{equation*}
\Vert h^{(6)}\Vert _{L^{2}\left( \mu \right) }\lesssim t^{6}e^{-c_{\gamma
}t}\Vert f_{0}\Vert _{L^{2}(\mu )}.
\end{equation*}

Next, we prove the estimate for the first $x$-derivative of $h^{(6)}$.
Notice that%
\begin{align*}
& \quad \nabla _{x}h^{(6)}(t) \\
& =\nabla
_{x}\int_{0}^{t}\int_{0}^{s_{1}}\int_{0}^{s_{2}}\int_{0}^{s_{3}}%
\int_{0}^{s_{4}}\int_{0}^{s_{5}}\frac{s_{1}-s_{2}}{s_{1}-s_{3}}\mathbb{S}%
^{t-s_{1}}K\mathbb{S}^{s_{1}-s_{2}}K\mathbb{S}^{s_{2}-s_{3}}K\mathbb{S}%
^{s_{3}-s_{4}}\,K\mathbb{S}^{s_{4}-s_{5}}K\mathbb{S}^{s_{5}-s_{6}}K\mathbb{S}%
^{s_{6}}f_{0}\,ds \\
& \quad +\nabla _{x}\int_{0}^{t}\int_{0}^{s_{1}}\int_{0}^{s_{2}}\frac{%
s_{2}-s_{3}}{s_{1}-s_{3}}\mathbb{S}^{t-s_{1}}K\mathbb{S}^{s_{1}-s_{2}}K%
\mathbb{S}^{s_{2}-s_{3}}K\mathbb{S}^{s_{3}-s_{4}}\,K\mathbb{S}^{s_{4}-s_{5}}K%
\mathbb{S}^{s_{5}-s_{6}}K\mathbb{S}^{s_{6}}f_{0}\,ds \\
&
=\int_{0}^{t}\int_{0}^{s_{1}}\int_{0}^{s_{2}}\int_{0}^{s_{3}}%
\int_{0}^{s_{4}}\int_{0}^{s_{5}}\frac{1}{s_{1}-s_{3}}A\left(
s_{1},s_{2},\ldots ,s_{6},x,\xi ,t\right) \,ds,
\end{align*}%
where $ds=ds_{6}ds_{5}ds_{4}ds_{3}ds_{2}ds_{1}$ and
\begin{align*}
& A\left( s_{1},s_{2},\ldots ,s_{6},x,\xi ,t\right) \\
& =\mathbb{S}^{t-s_{1}}K\left( \mathcal{D}_{s_{1}-s_{2}}-\nabla _{\xi
}\right) \mathbb{S}^{s_{1}-s_{2}}K\mathbb{S}^{s_{2}-s_{3}}K\mathbb{S}%
^{s_{3}-s_{4}}\,K\mathbb{S}^{s_{4}-s_{5}}K\mathbb{S}^{s_{5}-s_{6}}K\mathbb{S}%
^{s_{6}}f_{0} \\
& +\mathbb{S}^{t-s_{1}}K\mathbb{S}^{s_{1}-s_{2}}K\left( \mathcal{D}%
_{s_{2}-s_{3}}-\nabla _{\xi }\right) \mathbb{S}^{s_{2}-s_{3}}K\mathbb{S}%
^{s_{3}-s_{4}}\,K\mathbb{S}^{s_{4}-s_{5}}K\mathbb{S}^{s_{5}-s_{6}}K\mathbb{S}%
^{s_{6}}f_{0}.
\end{align*}%
From Lemma \ref{Operators-KSD}, it follows that
\begin{align*}
\left\Vert \nabla _{x}h^{(6)}(t)\right\Vert _{L^{2}(\mu )}& \lesssim
e^{-c_{\gamma
}t}\int_{0}^{t}\int_{0}^{s_{1}}\int_{0}^{s_{2}}\int_{0}^{s_{3}}%
\int_{0}^{s_{4}}\int_{0}^{s_{5}}\left( \frac{1}{s_{1}-s_{3}}+1\right)
\left\Vert f_{0}\right\Vert _{L^{2}(\mu )}ds \\
& \lesssim e^{-c_{\gamma }t}\left[ \int_{0}^{t}\int_{0}^{s_{1}}%
\int_{s_{3}}^{s_{1}}\frac{s_{3}^{3}}{s_{1}-s_{3}}ds_{2}ds_{3}ds_{1}+t^{6}%
\right] \left\Vert f_{0}\right\Vert _{L^{2}(\mu )} \\
& \lesssim \left( t^{6}+t^{5}\right) e^{-c_{\gamma }t}\Vert f_{0}\Vert
_{L^{2}(\mu )}\,.
\end{align*}%
Similarly, we have
\begin{align*}
& \nabla _{x}^{2}h^{(6)}(t) \\
& =\nabla
_{x}^{2}\int_{0}^{t}\int_{0}^{s_{1}}\int_{0}^{s_{2}}\int_{0}^{s_{3}}%
\int_{0}^{s_{4}}\int_{0}^{s_{5}} \\
& \left[ \frac{\left( s_{1}-s_{2}+s_{2}-s_{3}\right) }{s_{1}-s_{3}}\frac{%
(s_{4}-s_{5}+s_{5}-s_{6})}{s_{4}-s_{6}}\mathbb{S}^{t-s_{1}}K\mathbb{S}%
^{s_{1}-s_{2}}K\mathbb{S}^{s_{2}-s_{3}}K\mathbb{S}^{s_{3}-s_{4}}\,K\mathbb{S}%
^{s_{4}-s_{5}}K\mathbb{S}^{s_{5}-s_{6}}K\mathbb{S}^{s_{6}}f_{0}\,\right] ds
\\
&
=\int_{0}^{t}\int_{0}^{s_{1}}\int_{0}^{s_{2}}\int_{0}^{s_{3}}%
\int_{0}^{s_{4}}\int_{0}^{s_{5}}\frac{1}{(s_{1}-s_{3})(s_{4}-s_{6})}B\left(
s_{1},s_{2},\ldots ,s_{6},x,\xi ,t\right) \,ds
\end{align*}%
where $ds=ds_{6}ds_{5}ds_{4}ds_{3}ds_{2}ds_{1}$ and
\begin{align*}
& B\left( s_{1},s_{2},\ldots ,s_{6},x,\xi ,t\right) \\
& =\mathbb{S}^{t-s_{1}}K\left( \mathcal{D}_{s_{1}-s_{2}}-\nabla _{\xi
}\right) \mathbb{S}^{s_{1}-s_{2}}K\mathbb{S}^{s_{2}-s_{3}}K\mathbb{S}%
^{s_{3}-s_{4}}\,K\left( \mathcal{D}_{s_{4}-s_{5}}-\nabla _{\xi }\right)
\mathbb{S}^{s_{4}-s_{5}}K\mathbb{S}^{s_{5}-s_{6}}K\mathbb{S}^{s_{6}}f_{0} \\
& +\mathbb{S}^{t-s_{1}}K\mathbb{S}^{s_{1}-s_{2}}K\left( \mathcal{D}%
_{s_{1}-s_{2}}-\nabla _{\xi }\right) \mathbb{S}^{s_{2}-s_{3}}K\mathbb{S}%
^{s_{3}-s_{4}}\,K\mathbb{S}^{s_{4}-s_{5}}K\left( \mathcal{D}%
_{s_{5}-s_{6}}-\nabla _{\xi }\right) \mathbb{S}^{s_{5}-s_{6}}K\mathbb{S}%
^{s_{6}}f_{0} \\
& +\mathbb{S}^{t-s_{1}}K\mathbb{S}^{s_{1}-s_{2}}K\left( \mathcal{D}%
_{s_{2}-s_{3}}-\nabla _{\xi }\right) \mathbb{S}^{s_{2}-s_{3}}K\mathbb{S}%
^{s_{3}-s_{4}}\,K\left( \mathcal{D}_{s_{4}-s_{5}}-\nabla _{\xi }\right)
\mathbb{S}^{s_{4}-s_{5}}K\mathbb{S}^{s_{5}-s_{6}}K\mathbb{S}^{s_{6}}f_{0} \\
& +\mathbb{S}^{t-s_{1}}K\mathbb{S}^{s_{1}-s_{2}}K\left( \mathcal{D}%
_{s_{2}-s_{3}}-\nabla _{\xi }\right) \mathbb{S}^{s_{2}-s_{3}}K\mathbb{S}%
^{s_{3}-s_{4}}\,K\mathbb{S}^{s_{4}-s_{5}}K\left( \mathcal{D}%
_{s_{5}-s_{6}}-\nabla _{\xi }\right) \mathbb{S}^{s_{5}-s_{6}}K\mathbb{S}%
^{s_{6}}f_{0}\,.
\end{align*}%
By Lemma \ref{Operators-KSD} again, we deduce
\begin{align*}
& \quad \left\Vert \nabla _{x}^{2}h^{(6)}(t)\right\Vert _{L^{2}(\mu )} \\
& \lesssim e^{-c_{\gamma }t}\left\Vert f_{0}\right\Vert _{L^{2}(\mu
)}\int_{0}^{t}\int_{0}^{s_{1}}\int_{0}^{s_{2}}\int_{0}^{s_{3}}%
\int_{0}^{s_{4}}\int_{0}^{s_{5}}\left( 1+\frac{1}{s_{1}-s_{3}}\right) \left(
1+\frac{1}{s_{4}-s_{6}}\right) ds \\
& \lesssim e^{-c_{\gamma }t}\left( t^{6}+t^{5}+t^{4}\right) \Vert f_{0}\Vert
_{L^{2}(\mu )}\,.
\end{align*}%
Therefore, $\Vert h^{(6)}\Vert _{H_{x}^{2}L_{\xi }^{2}\left( \mu \right)
}\lesssim t^{4}\left( 1+t\right) ^{2}e^{-c_{\gamma }t}\Vert f_{0}\Vert
_{L^{2}(\mu )}\,.$
\end{proof}

As a consequence, owing to (\ref{R^(6)}) and Lemma \ref{Regularity}, we find
\begin{equation}
\Vert \mathcal{R}^{(6)}\Vert _{H_{x}^{2}L_{\xi }^{2}(\mu)}\leq
\int_{0}^{t}\Vert h^{(6)}\left( s\right) \Vert _{H_{x}^{2}L_{\xi
}^{2}(\mu)}ds\lesssim\left\{
\begin{array}{ll}
\left \{t^{5}\wedge 1\}\Vert f_{0}\right\Vert _{L^{2}(\mu) },%
\vspace {3mm}
& 0\leq \gamma <1, \\
t^{5}(1+t)^{2}\left\Vert f_{0}\right\Vert _{L^{2} (\mu)}, & -2<\gamma <0.%
\end{array}%
\right.  \label{R^(6)-1}
\end{equation}%
Here $\{a\wedge b\}=\min\{a,b\}$.


\section{Global wave structures}

\label{Global}

In this section we will complete the proof of Theorem \ref{thm:main} by
discussing the global wave structures inside the finite Mach number region
and outside the finite Mach number region separately.

\subsection{Inside the finite Mach number region}

By the long wave-short wave decomposition and wave-remainder decomposition,
we have
\begin{equation*}
f=f_{L}+f_{S}=W^{(6)}+\mathcal{R}^{(6)}\,.
\end{equation*}%
We now define the tail part as $f_{R}=\mathcal{R}^{(6)}-f_{L}=f_{S}-W^{(6)}$%
. Therefore $f$ can be rewritten as $f=f_{L}+W^{\left( 6\right) }+f_{R}.$

From Propositions \ref{LS-estimate1}-\ref{LS-estimate2} and Lemma \ref%
{pointwise-h}, the pointwise estimate of the long wave part $f_{L}$ and the
wave part $W^{(6)}$ are completed. It remains to study the tail part $f_{R}$%
. It is easy to see that
\begin{equation*}
\Vert f_{R}\Vert _{_{H_{x}^{2}L_{\xi }^{2}}}=\Vert (\mathcal{R}%
^{(6)}-f_{L})\Vert _{H_{x}^{2}L_{\xi }^{2}}\lesssim \left\{
\begin{array}{l}
\displaystyle\Vert f_{0}\Vert _{L^{2}},\quad \hbox{for}\quad 0\leq \gamma
<1\,, \\
\\
\displaystyle(1+t)^{7}\Vert f_{0}\Vert _{L^{2}},\quad \hbox{for}\quad
-2<\gamma <0,%
\end{array}%
\right.
\end{equation*}%
duo to (\ref{R^(6)-1}), and using Propositions \ref{LS-estimate1}-\ref%
{LS-estimate2} and Lemma \ref{L2-h} gives
\begin{equation*}
\Vert f_{R}\Vert _{L^{2}}=\Vert f_{S}-W^{(6)}\Vert _{L^{2}}\lesssim \left\{
\begin{array}{l}
\displaystyle e^{-Ct}\Vert f_{0}\Vert _{L^{2}}\,,\quad \hbox{for}\quad 0\leq
\gamma <1\,, \\
\\
\displaystyle e^{-c\alpha ^{\frac{-\gamma }{p-\gamma }}t^{\frac{p}{p-\gamma }%
}}\left\Vert f_{0}\right\Vert _{L^{2}\left( e^{7\alpha \left\vert \xi
\right\vert ^{p}}\right) }\quad \hbox{for}\quad -2<\gamma <0\,,%
\end{array}%
\right.
\end{equation*}%
for some constants $C,c>0$. The Sobolev inequality \cite[Theorem 5.8]%
{[Adams]} implies
\begin{equation}
|f_{R}|_{L_{\xi }^{2}}\leq \left\Vert f_{R}\right\Vert _{L_{\xi
}^{2}L_{x}^{\infty }}\lesssim \Vert f_{R}\Vert _{H_{x}^{2}L_{\xi
}^{2}}^{3/4}\Vert f_{R}\Vert _{L^{2}}^{1/4}\lesssim \left\{
\begin{array}{l}
\displaystyle e^{-\frac{1}{4}Ct}\Vert f_{0}\Vert _{L^{2}}\,,\quad \hbox{for}%
\quad 0\leq \gamma <1\,, \\
\\
\displaystyle e^{-\frac{1}{8}c\alpha ^{\frac{-\gamma }{p-\gamma }}t^{\frac{p%
}{p-\gamma }}}\left\Vert f_{0}\right\Vert _{L^{2}\left( e^{7\alpha
\left\vert \xi \right\vert ^{p}}\right) },\quad \hbox{for}\quad -2<\gamma
<0\,.%
\end{array}%
\right.  \label{f_{R}}
\end{equation}

Combining Propositions \ref{LS-estimate1}-\ref{LS-estimate2}, Lemma \ref%
{pointwise-h} and (\ref{f_{R}}), we obtain the pointwise estimate for the
solution inside the finite Mach number region.

\begin{proposition}
Let $f$ be the solution to the linearized Boltzmann equation \eqref{bot.1.d}
and let $\mathbf{v}=\sqrt{5/3}$ be the sound speed associated with the
normalized global Maxwellian. Then

\begin{enumerate}
\item As $0\leq \gamma <1$, for any given positive integer $N$, and any
given $0<p\leq 2$, $\beta>3/2$, sufficiently small $\alpha>0$, there exist
positive constants $C_{N}$, $C$ and $c_{0}$ such that
\begin{equation*}
\left\vert f(t,x,\cdot )\right\vert _{L_{\xi }^{2}}\leq C_{N}\left[
\begin{array}{l}
\left( 1+t\right) ^{-2}\left( 1+\frac{\left( \left\vert x\right\vert -%
\mathbf{v}t\right) ^{2}}{1+t}\right) ^{-N}+\left( 1+t\right) ^{-3/2}\left( 1+%
\frac{\left\vert x\right\vert ^{2}}{1+t}\right) ^{-N} \\[2mm]
+\mathbf{1}_{\{\left\vert x\right\vert \leq \mathbf{v}t\}}\left( 1+t\right)
^{-3/2}\left( 1+\frac{\left\vert x\right\vert ^{2}}{1+t}\right) ^{-3/2} \\%
[2mm]
+e^{-c_{0}\left( t+\alpha ^{\frac{1-\gamma }{p+1-\gamma }}\left\vert
x\right\vert ^{\frac{p}{p+1-\gamma }}\right) }+e^{-t/C}%
\end{array}%
\right] ||f_{0}||_{I}\,.
\end{equation*}

\item As $-2<\gamma <0$, for any given $0<p\leq 2$, $\beta>3/2$ and
sufficiently small $\alpha>0$, there exist positive constants $C$, $c$ and $%
c_{0}$ such that
\begin{equation*}
\left\vert f(t,x,\cdot )\right\vert _{L_{\xi }^{2}}\leq C\left[
\begin{array}{l}
\left( 1+t\right) ^{-3/2}+e^{-c\alpha ^{\frac{-\gamma }{p-\gamma }}t^{\frac{p%
}{p-\gamma }}} \\[2mm]
+e^{-c_{0}\left( \alpha ^{\frac{-\gamma }{p-\gamma }}t^{\frac{p}{p-\gamma }%
}+\alpha ^{\frac{1-\gamma }{p+1-\gamma }}\left\vert x\right\vert ^{\frac{p}{%
p+1-\gamma }}\right) }%
\end{array}%
\right] ||f_{0}||_{I}\,.
\end{equation*}
\end{enumerate}

Here $\mathbf{1}_{\{\cdot \}}$ is the indicator function and
\begin{equation*}
||f_{0}||_{I}\equiv \max \left\{ \left\Vert f_{0}\right\Vert _{L^{2}\left(
e^{7\alpha \left\vert \xi \right\vert ^{p}}\right) },\left\Vert
f_{0}\right\Vert _{L_{x}^{1}L_{\xi }^{2}},\left\Vert f_{0}\right\Vert
_{L_{x}^{\infty }L_{\xi }^{\infty }\left( e^{7\alpha \left\vert \xi
\right\vert ^{p}}\left<\xi\right>^{\beta}\right) }\right\} .
\end{equation*}
\end{proposition}

\subsection{Outside the finite Mach number region}

In the previous section we have well investigated the pointwise behavior for
the wave part $W^{(6)}$ (see Lemma \ref{pointwise-h}). To clarify the wave
structure outside the finite Mach number region, we still need to estimate
the remainder part $\mathcal{R}^{(6)}$. Here, the weighted energy estimate
plays a decisive role; especially, we have to refine the regularization
estimate by the domain decomposition. Consider the weight
\begin{equation}
w\left( t,x,\xi \right) =\exp \left( \epsilon \rho \left( t,x,\xi \right)
/2\right) ,  \label{weight}
\end{equation}%
with%
\begin{align*}
\rho \left( t,x,\xi \right) & =5\left( \delta \left( \left\langle
x\right\rangle -Mt\right) \right) ^{\frac{p}{p+1-\gamma }}\left( 1-\chi
\left( \frac{\delta \left( \left\langle x\right\rangle -Mt\right) }{%
\left\langle \xi \right\rangle ^{p+1-\gamma }}\right) \right) \\
& +\left[ \left( 1-\chi \left( \frac{\delta \left( \left\langle
x\right\rangle -Mt\right) }{\left\langle \xi \right\rangle ^{p+1-\gamma }}%
\right) \right) \left[ \delta \left( \left\langle x\right\rangle -Mt\right) %
\right] \left\langle \xi \right\rangle ^{\gamma -1}+3\left\langle \xi
\right\rangle ^{p}\right] \chi \left( \frac{\delta \left( \left\langle
x\right\rangle -Mt\right) }{\left\langle \xi \right\rangle ^{p+1-\gamma }}%
\right) .
\end{align*}%
where $\epsilon ,$ $\delta >0$ will be chosen sufficiently small and $M>0$
large enough later on. We define
\begin{equation*}
H_{+}=\{\left( x,\xi \right) :\delta \left( \left\langle x\right\rangle
-Mt\right) >2\left\langle \xi \right\rangle ^{p+1-\gamma }\},
\end{equation*}%
\begin{equation*}
H_{0}=\{\left( x,\xi \right) :\left\langle \xi \right\rangle ^{p+1-\gamma
}\leq \delta \left( \left\langle x\right\rangle -Mt\right) \leq
2\left\langle \xi \right\rangle ^{p+1-\gamma }\},
\end{equation*}%
and%
\begin{equation*}
H_{-}=\{\left( x,\xi \right) :\delta \left( \left\langle x\right\rangle
-Mt\right) <\left\langle \xi \right\rangle ^{p+1-\gamma }\}.
\end{equation*}%
To go further, we need to estimate $\left\vert \int \left\langle g,\left(
K_{\epsilon }-K\right) g\right\rangle _{\xi }dx\right\vert $, where $%
K_{\epsilon }=e^{\epsilon \rho \left( t,x,\xi \right) }Ke^{-\epsilon \rho
\left( t,x,\xi \right) }.$ This estimate will be used in the weighted energy
estimate of $\mathcal{R}^{(6)}$ (Proposition \ref{weighted-remainder}). For
simplicity of notations, let $\mathrm{P}_{0}g=\sum_{j=0}^{4}b_{j}\chi _{j}$,
$b_{j}=\left\langle g,\chi _{j}\right\rangle _{\xi }$.

\begin{lemma}
\label{K-epsilon}Let $0<p\leq2$. There exists a constant $C=C\left(
\gamma,p\right) >0$ such that for any $0<\epsilon\ll1$,
\begin{align}
\left\vert \int\left\langle g,\left( K_{\epsilon}-K\right) g\right\rangle
_{\xi}dx\right\vert & \leq C\epsilon\int\left\langle \xi\right\rangle
^{\gamma}\left\vert \mathrm{P}_{1}g\right\vert ^{2}d\xi dx  \notag \\
& +C\epsilon\left[ \int_{H_{+}}\left[ \delta\left( \left\langle
x\right\rangle -Mt\right) \right] ^{\frac{\gamma-1}{p+1-\gamma}}\left\vert
\mathrm{P}_{0}g\right\vert ^{2}d\xi dx+\int_{H_{0}\cup H_{-}}\left\vert
\mathrm{P}_{0}g\right\vert ^{2}d\xi dx\right] .  \label{K-epsilon-1}
\end{align}
Consequently,%
\begin{align}
\int\left\langle g,L_{\epsilon}g\right\rangle _{\xi}dx & \leq\int
\left\langle g,Lg\right\rangle _{\xi}dx+C\epsilon\int\left\langle
\xi\right\rangle ^{\gamma}\left\vert \mathrm{P}_{1}g\right\vert ^{2}d\xi dx
\notag \\
& +C\epsilon\left[ \int_{H_{+}}\left[ \delta\left( \left\langle
x\right\rangle -Mt\right) \right] ^{\frac{\gamma-1}{p+1-\gamma}}\left\vert
\mathrm{P}_{0}g\right\vert ^{2}d\xi dx+\int_{H_{0}\cup H_{-}}\left\vert
\mathrm{P}_{0}g\right\vert ^{2}d\xi dx\right] ,  \label{K-epsilon-2}
\end{align}
where $L_{\epsilon}=\allowbreak e^{\epsilon\rho\left( t,x,\xi\right)
}Le^{-\epsilon\rho\left( t,x,\xi\right) }.$
\end{lemma}

\begin{proof}
We split the integral
\begin{align}
\int \left\langle g,\left( K_{\epsilon }-K\right) g\right\rangle _{\xi }dx&
=\int \left\langle \mathrm{P}_{1}g,\left( K_{\epsilon }-K\right) \mathrm{P}%
_{1}g\right\rangle _{\xi }dx+\int \left\langle \mathrm{P}_{0}g,\left(
K_{\epsilon }-K\right) \mathrm{P}_{0}g\right\rangle _{\xi }dx  \label{K-g} \\
& +\int \left\langle \mathrm{P}_{1}g,\left( K_{\epsilon }-K\right) \mathrm{P}%
_{0}g\right\rangle _{\xi }dx+\int \left\langle \mathrm{P}_{1}g,\left(
K_{-\epsilon }-K\right) \mathrm{P}_{0}g\right\rangle _{\xi }dx.  \notag
\end{align}%
Firstly, following the same procedure as in Lemma \ref{K-esti-2}, we have
that for any $\epsilon >0$ sufficiently small,
\begin{equation}
\left\vert \int \left\langle \mathrm{P}_{1}g,\left( K_{\epsilon }-K\right)
\mathrm{P}_{1}g\right\rangle _{\xi }dx\right\vert \lesssim \epsilon
\left\Vert \mathrm{P}_{1}g\right\Vert _{L_{\sigma }^{2}}^{2}.  \label{K-P1P1}
\end{equation}%
Next, we estimate $\int \left\langle \mathrm{P}_{0}g,\left( K_{\epsilon
}-K\right) \mathrm{P}_{0}g\right\rangle _{\xi }dx.$\newline
\textbf{Estimate on }$\int \left\langle \mathrm{P}_{0}g,\left( K_{\epsilon
}-K\right) \mathrm{P}_{0}g\right\rangle _{\xi }dx.$ We split the integral
\begin{align*}
& \int \left\langle \mathrm{P}_{0}g,\left( K_{\epsilon }-K\right) \mathrm{P}%
_{0}g\right\rangle _{\xi }dx \\
& =\left[
\begin{array}{c}
\int_{\delta \left( \left\langle x\right\rangle -Mt\right) >0}\int_{\delta
\left( \left\langle x\right\rangle -Mt\right) >2\left\langle \xi
\right\rangle ^{p+1-\gamma }}\int_{\delta \left( \left\langle x\right\rangle
-Mt\right) \leq 2\left\langle \xi _{\ast }\right\rangle ^{p+1-\gamma }} \\
+\int_{\delta \left( \left\langle x\right\rangle -Mt\right) >0}\int_{\delta
\left( \left\langle x\right\rangle -Mt\right) \leq 2\left\langle \xi
\right\rangle ^{p+1-\gamma }}\int_{\delta \left( \left\langle x\right\rangle
-Mt\right) \geq 2\left\langle \xi _{\ast }\right\rangle ^{p+1-\gamma }} \\
+\int_{\delta \left( \left\langle x\right\rangle -Mt\right) >0}\int_{\delta
\left( \left\langle x\right\rangle -Mt\right) \leq 2\left\langle \xi
\right\rangle ^{p+1-\gamma }}\int_{\delta \left( \left\langle x\right\rangle
-Mt\right) \leq 2\left\langle \xi _{\ast }\right\rangle ^{p+1-\gamma }} \\
+\int_{\delta \left( \left\langle x\right\rangle -Mt\right) \leq
0}\int_{\delta \left( \left\langle x\right\rangle -Mt\right) \leq
2\left\langle \xi \right\rangle ^{p+1-\gamma }}\int_{\delta \left(
\left\langle x\right\rangle -Mt\right) \leq 2\left\langle \xi _{\ast
}\right\rangle ^{p+1-\gamma }}%
\end{array}%
\right] \mathrm{P}_{0}g\left( \xi \right) k\left( \xi ,\xi _{\ast }\right)
A_{\epsilon }\left( t,x,\xi ,\xi _{\ast }\right) \mathrm{P}_{0}g\left( \xi
_{\ast }\right) d\xi _{\ast }d\xi dx \\
& =I_{a}+I_{b}+I_{c}+II,
\end{align*}%
where $A_{\epsilon }\left( t,x,\xi ,\xi _{\ast }\right) =\left[ e^{\epsilon
\left( \rho \left( t,x,\xi \right) -\rho \left( t,x,\xi _{\ast }\right)
\right) }-1\right] .$ We remark here that $A_{\epsilon }\left( t,x,\xi ,\xi
_{\ast }\right) =0$ whenever $\delta \left( \left\langle x\right\rangle
-Mt\right) >2\left\langle \xi \right\rangle ^{p+1-\gamma }$ and $\delta
\left( \left\langle x\right\rangle -Mt\right) >2\left\langle \xi _{\ast
}\right\rangle ^{p+1-\gamma }$; in other words, there is no contribution to
the integral in this region. Note that $\rho \left( t,x,\xi \right) $ also
satisfies
\begin{equation*}
\left\vert \rho \left( t,x,\xi \right) -\rho \left( t,x,\xi _{\ast }\right)
\right\vert \leq c_{1}\left\vert \left\vert \xi \right\vert ^{2}-\left\vert
\xi _{\ast }\right\vert ^{2}\right\vert ,
\end{equation*}%
whose proof is similar to Lemma \ref{K-esti-2}, hence%
\begin{equation*}
\left\vert A_{\epsilon }\left( t,x,\xi ,\xi _{\ast }\right) \right\vert
\lesssim \epsilon \left\vert \left\vert \xi \right\vert ^{2}-\left\vert \xi
_{\ast }\right\vert ^{2}\right\vert e^{c_{1}\epsilon \left\vert \left\vert
\xi \right\vert ^{2}-\left\vert \xi _{\ast }\right\vert ^{2}\right\vert }.
\end{equation*}

Now, for $I_{a},$ we have
\begin{equation*}
\left\vert A_{\epsilon }\left( t,x,\xi ,\xi _{\ast }\right) \mathrm{P}%
_{0}g\left( \xi _{\ast }\right) \right\vert \lesssim \epsilon \left\vert
\left\vert \xi \right\vert ^{2}-\left\vert \xi _{\ast }\right\vert
^{2}\right\vert e^{c_{1}\epsilon \left\vert \left\vert \xi \right\vert
^{2}-\left\vert \xi _{\ast }\right\vert ^{2}\right\vert }\left(
\sum_{j=0}^{4}\left\vert \chi _{j}\left( \xi _{\ast }\right) \right\vert
^{2}\right) ^{1/2}\left( \sum_{j=0}^{4}b_{j}^{2}\right) ^{1/2},
\end{equation*}%
and thus
\begin{align*}
& \int_{\delta \left( \left\langle x\right\rangle -Mt\right) >2\left\langle
\xi \right\rangle ^{p+1-\gamma }}\int_{\delta \left( \left\langle
x\right\rangle -Mt\right) \leq 2\left\langle \xi _{\ast }\right\rangle
^{p+1-\gamma }}\left\vert \mathrm{P}_{0}g\left( \xi \right) k\left( \xi ,\xi
_{\ast }\right) A_{\epsilon }\left( t,x,\xi ,\xi _{\ast }\right) \mathrm{P}%
_{0}g\left( \xi _{\ast }\right) \right\vert d\xi _{\ast }d\xi \\
& \lesssim \epsilon \int_{\delta \left( \left\langle x\right\rangle
-Mt\right) >2\left\langle \xi \right\rangle ^{p+1-\gamma }}\left\vert
\mathrm{P}_{0}g\left( \xi \right) \right\vert \left(
\sum_{j=0}^{4}b_{j}^{2}\right) ^{1/2}\int_{\delta \left( \left\langle
x\right\rangle -Mt\right) \leq 2\left\langle \xi _{\ast }\right\rangle
^{p+1-\gamma }}\left\vert k\left( \xi ,\xi _{\ast }\right) \right\vert
e^{-c\left\vert \xi _{\ast }\right\vert ^{2}}d\xi _{\ast } \\
& \lesssim \epsilon \int_{\delta \left( \left\langle x\right\rangle
-Mt\right) >2\left\langle \xi \right\rangle ^{p+1-\gamma }}\left(
\sum_{j=0}^{4}\left\vert \chi _{j}\left( \xi \right) \right\vert ^{2}\right)
^{1/2}\left( \sum_{j=0}^{4}b_{j}^{2}\right) \exp \left( -c^{\prime }\left[
\delta \left( \left\langle x\right\rangle -Mt\right) \right] ^{\frac{2}{%
p+1-\gamma }}\right) d\xi \\
& \lesssim \epsilon \exp \left( -c^{\prime }\left[ \delta \left(
\left\langle x\right\rangle -Mt\right) \right] ^{\frac{2}{p+1-\gamma }%
}\right) \int_{\delta \left( \left\langle x\right\rangle -Mt\right)
>2\left\langle \xi \right\rangle ^{p+1-\gamma }}\left\vert \mathrm{P}%
_{0}g\right\vert ^{2}d\xi \\
& +\epsilon \int_{\delta \left( \left\langle x\right\rangle -Mt\right) \leq
2\left\langle \xi \right\rangle ^{p+1-\gamma }}\left\vert \mathrm{P}%
_{0}g\right\vert ^{2}d\xi .
\end{align*}%
The first inequality is valid since $\left\vert \xi \right\vert <\left\vert
\xi _{\ast }\right\vert $ and $\left( \sum_{j=0}^{4}\left\vert \chi
_{j}\left( \xi \right) \right\vert ^{2}\right) ^{1/2}$ decays exponentially;
the second inequality holds due to the fact that $e^{-c\left\vert \xi _{\ast
}\right\vert ^{2}}\lesssim \exp \left( -c^{\prime }\left[ \delta \left(
\left\langle x\right\rangle -Mt\right) \right] ^{\frac{2}{p+1-\gamma }%
}\right) $ for some constant $c^{\prime }>0$ whenever $\delta \left(
\left\langle x\right\rangle -Mt\right) \leq 2\left\langle \xi _{\ast
}\right\rangle ^{p+1-\gamma }$ and that $\left\vert k\left( \xi ,\cdot
\right) \right\vert $ is integrable. Hence,
\begin{equation*}
\left\vert I_{a}\right\vert \lesssim \epsilon \left[ \int_{H_{+}}\left[
\left( \delta \left( \left\langle x\right\rangle -Mt\right) \right) \right]
^{\frac{\gamma -1}{p+1-\gamma }}\left\vert \mathrm{P}_{0}g\right\vert
^{2}d\xi dx+\int_{H_{0}\cup H_{-}}\left\vert \mathrm{P}_{0}g\right\vert
^{2}d\xi dx\right] .
\end{equation*}%
Similarly for $I_{b},$ it follows
\begin{equation*}
\left\vert I_{b}\right\vert \lesssim \epsilon \left[ \int_{H_{+}}\left[
\left( \delta \left( \left\langle x\right\rangle -Mt\right) \right) \right]
^{\frac{\gamma -1}{p+1-\gamma }}\left\vert \mathrm{P}_{0}g\right\vert
^{2}d\xi dx+\int_{H_{0}\cup H_{-}}\left\vert \mathrm{P}_{0}g\right\vert
^{2}d\xi dx\right] .
\end{equation*}%
On the other hand, by symmetry%
\begin{align*}
\left\vert I_{c}\right\vert & \leq 2c_{1}\epsilon \int_{\delta \left(
\left\langle x\right\rangle -Mt\right) >0}\int_{\delta \left( \left\langle
x\right\rangle -Mt\right) \leq 2\left\langle \xi \right\rangle ^{p+1-\gamma
}}\int_{\delta \left( \left\langle x\right\rangle -Mt\right) \leq
2\left\langle \xi _{\ast }\right\rangle ^{p+1-\gamma },\left\vert \xi
\right\vert <\left\vert \xi _{\ast }\right\vert } \\
& \left\vert \mathrm{P}_{0}g\left( \xi \right) k\left( \xi ,\xi _{\ast
}\right) \mathrm{P}_{0}g\left( \xi _{\ast }\right) \right\vert \left\vert
\left\vert \xi \right\vert ^{2}-\left\vert \xi _{\ast }\right\vert
^{2}\right\vert e^{c_{1}\epsilon \left\vert \left\vert \xi \right\vert
^{2}-\left\vert \xi _{\ast }\right\vert ^{2}\right\vert }d\xi _{\ast }d\xi
dx,
\end{align*}%
applying a similar argument for $I_{a}$ gives%
\begin{equation*}
\left\vert I_{c}\right\vert \lesssim \epsilon \left[ \int_{H_{+}}\left[
\left( \delta \left( \left\langle x\right\rangle -Mt\right) \right) \right]
^{\frac{\gamma -1}{p+1-\gamma }}\left\vert \mathrm{P}_{0}g\right\vert
^{2}d\xi dx+\int_{H_{0}\cup H_{-}}\left\vert \mathrm{P}_{0}g\right\vert
^{2}d\xi dx\right] ,
\end{equation*}%
as well.

Following the same argument as the proof of $\left( \ref{Weighted-K-1}%
\right) $ in Lemma \ref{K-esti-2}, it is easy to see that%
\begin{align*}
\left\vert II\right\vert & \lesssim \epsilon \int_{\delta \left(
\left\langle x\right\rangle -Mt\right) \leq 0}\int \left\vert \mathrm{P}%
_{0}g\right\vert ^{2}d\xi dx \\
& =\epsilon \int_{\delta \left( \left\langle x\right\rangle -Mt\right) \leq
0}\int_{\delta \left( \left\langle x\right\rangle -Mt\right) \leq
2\left\langle \xi \right\rangle ^{p+1-\gamma }}\left\vert \mathrm{P}%
_{0}g\right\vert ^{2}d\xi dx \\
& \leq \epsilon \int_{H_{0}\cup H_{-}}\left\vert \mathrm{P}_{0}g\right\vert
^{2}d\xi dx.
\end{align*}%
Gathering the estimates for $I_{a}$, $I_{b},$ $I_{c}$ and $II$ yields
\begin{equation}
\left\vert \int \left\langle \mathrm{P}_{0}g,\left( K_{\epsilon }-K\right)
\mathrm{P}_{0}g\right\rangle _{\xi }dx\right\vert \lesssim \epsilon \left[
\int_{H_{+}}\left[ \left( \delta \left( \left\langle x\right\rangle
-Mt\right) \right) \right] ^{\frac{\gamma -1}{p+1-\gamma }}\left\vert
\mathrm{P}_{0}g\right\vert ^{2}d\xi dx+\int_{H_{0}\cup H_{-}}\left\vert
\mathrm{P}_{0}g\right\vert ^{2}d\xi dx\right] .  \label{K-P0P0}
\end{equation}%
\newline
\textbf{Estimate on }$\int \left\langle \mathrm{P}_{1}g,\left( K_{\epsilon
}-K\right) \mathrm{P}_{0}g\right\rangle _{\xi }dx+\int \left\langle \mathrm{P%
}_{1}g,\left( K_{-\epsilon }-K\right) \mathrm{P}_{0}g\right\rangle _{\xi
}dx. $ We split the integral%
\begin{align*}
& \int \left\langle \mathrm{P}_{1}g,\left( K_{\epsilon }-K\right) \mathrm{P}%
_{0}g\right\rangle _{\xi }dx+\int \left\langle \mathrm{P}_{1}g,\left(
K_{-\epsilon }-K\right) \mathrm{P}_{0}g\right\rangle _{\xi }dx \\
& =\left[
\begin{array}{c}
\int_{\delta \left( \left\langle x\right\rangle -Mt\right) >0}\int_{\delta
\left( \left\langle x\right\rangle -Mt\right) >2\left\langle \xi
\right\rangle ^{p+1-\gamma }}\int_{\delta \left( \left\langle x\right\rangle
-Mt\right) \leq 2\left\langle \xi _{\ast }\right\rangle ^{p+1-\gamma }} \\
+\int_{\delta \left( \left\langle x\right\rangle -Mt\right) >0}\int_{\delta
\left( \left\langle x\right\rangle -Mt\right) \leq 2\left\langle \xi
\right\rangle ^{p+1-\gamma }}\int_{\delta \left( \left\langle x\right\rangle
-Mt\right) \geq 2\left\langle \xi _{\ast }\right\rangle ^{p+1-\gamma }} \\
+\int_{\delta \left( \left\langle x\right\rangle -Mt\right) >0}\int_{\delta
\left( \left\langle x\right\rangle -Mt\right) \leq 2\left\langle \xi
\right\rangle ^{p+1-\gamma }}\int_{\delta \left( \left\langle x\right\rangle
-Mt\right) \leq 2\left\langle \xi _{\ast }\right\rangle ^{p+1-\gamma }} \\
+\int_{\delta \left( \left\langle x\right\rangle -Mt\right) \leq
0}\int_{\delta \left( \left\langle x\right\rangle -Mt\right) \leq
2\left\langle \xi \right\rangle ^{p+1-\gamma }}\int_{\delta \left(
\left\langle x\right\rangle -Mt\right) \leq 2\left\langle \xi _{\ast
}\right\rangle ^{p+1-\gamma }}%
\end{array}%
\right] \mathrm{P}_{1}g\left( \xi \right) k\left( \xi ,\xi _{\ast }\right)
B_{\epsilon }\left( t,x,\xi ,\xi _{\ast }\right) \mathrm{P}_{0}g\left( \xi
_{\ast }\right) d\xi _{\ast }d\xi dx \\
& =I^{a}+I^{b}+I^{c}+II^{\prime },
\end{align*}%
where $B_{\epsilon }\left( t,x,\xi ,\xi _{\ast }\right) =A_{\epsilon }\left(
t,x,\xi ,\xi _{\ast }\right) +A_{\epsilon }\left( t,x,\xi _{\ast },\xi
\right) .$ It readily follows from the definition of $B_{\epsilon }\ $that
\begin{equation}
\left\vert B_{\epsilon }\left( t,x,\xi ,\xi _{\ast }\right) \right\vert
\lesssim \epsilon \left\vert \left\vert \xi \right\vert ^{2}-\left\vert \xi
_{\ast }\right\vert ^{2}\right\vert e^{c_{1}\epsilon \left\vert \left\vert
\xi \right\vert ^{2}-\left\vert \xi _{\ast }\right\vert ^{2}\right\vert }.
\label{B-epsilon}
\end{equation}

According the above discussion, we obtain%
\begin{align*}
\left\vert I^{a}\right\vert & \lesssim\epsilon\int_{\delta\left(
\left\langle x\right\rangle -Mt\right) >0}\int_{\delta\left( \left\langle
x\right\rangle -Mt\right) >2\left\langle \xi\right\rangle
^{p+1-\gamma}}\left\vert \mathrm{P}_{1}g \left( \xi\right) \right\vert
\exp\left( -c^{\prime }\left( \delta\left( \left\langle x\right\rangle
-Mt\right) \right) ^{\frac{2}{p+1-\gamma}}\right) \left(
\sum_{j=0}^{4}b_{j}^{2}\right) ^{1/2}d\xi dx \\
& \lesssim\epsilon\int\left\langle \xi\right\rangle ^{\gamma}\left\vert
\mathrm{P}_{1}g \right\vert ^{2}d\xi dx \\
& +\epsilon\int_{\delta\left( \left\langle x\right\rangle -Mt\right)
>0}\int_{\delta\left( \left\langle x\right\rangle -Mt\right) >2\left\langle
\xi\right\rangle ^{p+1-\gamma}}\left\langle \xi\right\rangle
^{-\gamma}\exp\left( -2c^{\prime}\left( \delta\left( \left\langle
x\right\rangle -Mt\right) \right) ^{\frac{2}{p+1-\gamma}}\right) \left(
\sum_{j=0}^{4}b_{j}^{2}\right) d\xi dx.
\end{align*}
For $0\leq\gamma<1,$
\begin{align*}
& \int_{\delta\left( \left\langle x\right\rangle -Mt\right) >0}\int
_{\delta\left( \left\langle x\right\rangle -Mt\right) >2\left\langle
\xi\right\rangle ^{p+1-\gamma}}\left\langle \xi\right\rangle
^{-\gamma}\exp\left( -2c^{\prime}\left( \delta\left( \left\langle
x\right\rangle -Mt\right) \right) ^{\frac{2}{p+1-\gamma}}\right) \left(
\sum_{j=0}^{4}b_{j}^{2}\right) d\xi dx \\
& \lesssim\int_{\delta\left( \left\langle x\right\rangle -Mt\right) >0}\left[
\left( \delta\left( \left\langle x\right\rangle -Mt\right) \right) \right] ^{%
\frac{3}{p+1-\gamma}}\exp\left( -2c^{\prime}\left( \delta\left( \left\langle
x\right\rangle -Mt\right) \right) ^{\frac {2}{p+1-\gamma}}\right) \left(
\sum_{j=0}^{4}b_{j}^{2}\right) dx \\
& \lesssim\int_{H_{+}}\left[ \left( \delta\left( \left\langle x\right\rangle
-Mt\right) \right) \right] ^{\frac{\gamma-1}{p+1-\gamma}}\left\vert \mathrm{P%
}_{0}g \right\vert ^{2}d\xi dx+\int_{H_{0}\cup H_{-}}\left\vert \mathrm{P}%
_{0}g \right\vert ^{2}d\xi dx,
\end{align*}
and for $-2<\gamma<0,$%
\begin{align*}
& \int_{\delta\left( \left\langle x\right\rangle -Mt\right) >0}\int
_{\delta\left( \left\langle x\right\rangle -Mt\right) >2\left\langle
\xi\right\rangle ^{p+1-\gamma}}\left\langle \xi\right\rangle
^{-\gamma}\exp\left( -2c^{\prime}\left( \delta\left( \left\langle
x\right\rangle -Mt\right) \right) ^{\frac{2}{p+1-\gamma}}\right) \left(
\sum_{j=0}^{4}b_{j}^{2}\right) d\xi dx \\
& \lesssim\int_{\delta\left( \left\langle x\right\rangle -Mt\right) >0}\left[
\left( \delta\left( \left\langle x\right\rangle -Mt\right) \right) \right] ^{%
\frac{3-\gamma}{p+1-\gamma}}\exp\left( -2c^{\prime }\left( \delta\left(
\left\langle x\right\rangle -Mt\right) \right) ^{\frac{2}{p+1-\gamma}%
}\right) \left( \sum_{j=0}^{4}b_{j}^{2}\right) dx \\
& \lesssim\int_{H_{+}}\left[ \left( \delta\left( \left\langle x\right\rangle
-Mt\right) \right) \right] ^{\frac{\gamma-1}{p+1-\gamma}}\left\vert \mathrm{P%
}_{0}g \right\vert ^{2}d\xi dx+\int_{H_{0}\cup H_{-}}\left\vert \mathrm{P}%
_{0}g \right\vert ^{2}d\xi dx.
\end{align*}
Hence, we conclude
\begin{equation*}
\left\vert I^{a}\right\vert \lesssim\epsilon\left( \int\left\langle
\xi\right\rangle ^{\gamma}\left\vert \mathrm{P}_{1}g \right\vert ^{2}d\xi
dx+\int_{H_{+}}\left[ \left( \delta\left( \left\langle x\right\rangle
-Mt\right) \right) \right] ^{\frac{\gamma-1}{p+1-\gamma}}\left\vert \mathrm{P%
}_{0}g \right\vert ^{2}d\xi dx+\int_{H_{0}\cup H_{-}}\left\vert \mathrm{P}%
_{0}g \right\vert ^{2}d\xi dx\right) .
\end{equation*}

For $II^{\prime },$ similar to Lemma \ref{K-esti-2},
\begin{align*}
\left\vert II^{\prime }\right\vert & \lesssim \epsilon \int_{\delta \left(
\left\langle x\right\rangle -Mt\right) \leq 0}\left\vert \left\langle \xi
\right\rangle ^{\gamma /2}\mathrm{P}_{1}g\right\vert _{L_{\xi
}^{2}}\left\vert \left\langle \xi \right\rangle ^{-\gamma /2}\mathrm{P}%
_{0}g\right\vert _{L_{\xi }^{2}}dx \\
& \lesssim \epsilon \int \left\langle \xi \right\rangle ^{\gamma }\left\vert
\mathrm{P}_{1}g\right\vert ^{2}d\xi dx+\epsilon \int_{\delta \left(
\left\langle x\right\rangle -Mt\right) \leq 0}\int_{\delta \left(
\left\langle x\right\rangle -Mt\right) \leq 2\left\langle \xi \right\rangle
^{p+1-\gamma }}\left\langle \xi \right\rangle ^{-\gamma }\left\vert \mathrm{P%
}_{0}g\right\vert ^{2}d\xi dx.
\end{align*}%
For $0\leq \gamma <1,$%
\begin{equation*}
\int_{\delta \left( \left\langle x\right\rangle -Mt\right) \leq
2\left\langle \xi \right\rangle ^{p+1-\gamma }}\left\langle \xi
\right\rangle ^{-\gamma }\left\vert \mathrm{P}_{0}g\right\vert ^{2}d\xi \leq
\int_{\delta \left( \left\langle x\right\rangle -Mt\right) \leq
2\left\langle \xi \right\rangle ^{p+1-\gamma }}\left\vert \mathrm{P}%
_{0}g\right\vert ^{2}d\xi ,
\end{equation*}%
and for $-2<\gamma <0,$%
\begin{eqnarray*}
\int_{\delta \left( \left\langle x\right\rangle -Mt\right) \leq
2\left\langle \xi \right\rangle ^{p+1-\gamma }}\left\langle \xi
\right\rangle ^{-\gamma }\left\vert \mathrm{P}_{0}g\right\vert ^{2}d\xi
&\leq &\int_{\delta \left( \left\langle x\right\rangle -Mt\right) \leq
2\left\langle \xi \right\rangle ^{p+1-\gamma }}\left\langle \xi
\right\rangle ^{-\gamma }\left( \sum_{j=0}^{4}\left\vert \chi _{j}\left( \xi
\right) \right\vert ^{2}\right) d\xi \left( \sum_{j=0}^{4}b_{j}^{2}\right) \\
&\lesssim &\left( \sum_{j=0}^{4}b_{j}^{2}\right) .
\end{eqnarray*}%
Hence, we deduce
\begin{align*}
\left\vert II^{\prime }\right\vert & \lesssim \epsilon \int \left\langle \xi
\right\rangle ^{\gamma }\left\vert \mathrm{P}_{1}g\right\vert ^{2}d\xi
dx+\epsilon \int_{\delta \left( \left\langle x\right\rangle -Mt\right) \leq
0}\int_{\delta \left( \left\langle x\right\rangle -Mt\right) \leq
2\left\langle \xi \right\rangle ^{p+1-\gamma }}\left\vert \mathrm{P}%
_{0}g\right\vert ^{2}d\xi dx \\
& \lesssim \epsilon \int \left\langle \xi \right\rangle ^{\gamma }\left\vert
\mathrm{P}_{1}g\right\vert ^{2}d\xi dx+\epsilon \int_{H_{0}\cup
H_{-}}\left\vert \mathrm{P}_{0}g\right\vert ^{2}d\xi dx.
\end{align*}

For $I^{b},$ we observe that $\left\vert \xi _{\ast }\right\vert \leq
\left\vert \xi \right\vert $ in this region. Further, note that%
\begin{align*}
& \quad k_{1}\left( \xi ,\xi _{\ast }\right) \left\vert B_{\epsilon }\left(
t,x,\xi ,\xi _{\ast }\right) \mathrm{P}_{0}g\left( \xi _{\ast }\right)
\right\vert \\
& \lesssim \left\vert \xi -\xi _{\ast }\right\vert ^{\gamma }\exp \left( -%
\frac{1}{4}\left\vert \xi \right\vert ^{2}-\frac{1}{4}\left\vert \xi _{\ast
}\right\vert ^{2}\right) \left\vert B_{\epsilon }\left( t,x,\xi ,\xi _{\ast
}\right) \mathrm{P}_{0}g\left( \xi _{\ast }\right) \right\vert \\
& =\widetilde{p}\left( \xi ,\xi _{\ast }\right) \times \left[ \exp \left( -%
\frac{1}{8}\left\vert \xi \right\vert ^{2}-\frac{1}{8}\left\vert \xi _{\ast
}\right\vert ^{2}\right) \left\vert B_{\epsilon }\left( t,x,\xi ,\xi _{\ast
}\right) \mathrm{P}_{0}g\left( \xi _{\ast }\right) \right\vert \right]
\end{align*}%
satisfies%
\begin{equation*}
\left\vert \exp \left( -\frac{1}{8}\left\vert \xi \right\vert ^{2}-\frac{1}{8%
}\left\vert \xi _{\ast }\right\vert ^{2}\right) B_{\epsilon }\left( t,x,\xi
,\xi _{\ast }\right) \mathrm{P}_{0}g\left( \xi _{\ast }\right) \right\vert
\lesssim \left[ \epsilon \exp \left( -\frac{1}{16}\left\vert \xi \right\vert
^{2}\right) \left( \sum_{j=0}^{4}b_{j}^{2}\right) ^{1/2}\right] \exp \left(
-c_{2}\left\vert \xi _{\ast }\right\vert ^{2}\right) ,
\end{equation*}%
and%
\begin{align*}
& k_{2}\left( \xi ,\xi _{\ast }\right) B_{\epsilon }\left( t,x,\xi ,\xi
_{\ast }\right) \mathrm{P}_{0}g\left( \xi _{\ast }\right) \\
& =p\left( \xi ,\xi _{\ast }\right) \left\{ \exp \left( -\frac{1}{16}\left[
\frac{\left( \left\vert \xi \right\vert ^{2}-\left\vert \xi _{\ast
}\right\vert ^{2}\right) ^{2}}{\left\vert \xi -\xi _{\ast }\right\vert ^{2}}%
+\left\vert \xi -\xi _{\ast }\right\vert ^{2}\right] \right) \times
B_{\epsilon }\left( t,x,\xi ,\xi _{\ast }\right) \mathrm{P}_{0}g\left( \xi
_{\ast }\right) \right\}
\end{align*}%
satisfies
\begin{align*}
& \left\vert \exp \left( -\frac{1}{16}\left[ \frac{\left( \left\vert \xi
\right\vert ^{2}-\left\vert \xi _{\ast }\right\vert ^{2}\right) ^{2}}{%
\left\vert \xi -\xi _{\ast }\right\vert ^{2}}+\left\vert \xi -\xi _{\ast
}\right\vert ^{2}\right] \right) B_{\epsilon }\left( t,x,\xi ,\xi _{\ast
}\right) \mathrm{P}_{0}g\left( \xi _{\ast }\right) \right\vert \\
& \lesssim \epsilon \exp \left( -\frac{1}{16}\left\vert \left\vert \xi
\right\vert ^{2}-\left\vert \xi _{\ast }\right\vert ^{2}\right\vert \right)
\left( \sum_{j=0}^{4}\left\vert \chi _{j}\left( \xi _{\ast }\right)
\right\vert ^{2}\right) ^{1/2}\left( \sum_{j=0}^{4}b_{j}^{2}\right) ^{1/2} \\
& \lesssim \left[ \epsilon \exp \left( -\frac{1}{16}\left\vert \xi
\right\vert ^{2}\right) \left( \sum_{j=0}^{4}b_{j}^{2}\right) ^{1/2}\right]
\exp \left( -c_{2}\left\vert \xi _{\ast }\right\vert ^{2}\right) ,
\end{align*}%
for some $c_{2}>0,$ where $\widetilde{p}\left( \xi ,\xi _{\ast }\right) $
and $p\left( \xi ,\xi _{\ast }\right) $ are kernels of bounded operators on $%
L_{\xi }^{2}.$ Since $e^{-c_{2}\left\vert \xi _{\ast }\right\vert ^{2}}\in
L_{\xi _{\ast }}^{2}$ and
\begin{equation*}
\left\langle \xi \right\rangle ^{-\frac{\gamma }{2}}e^{-\frac{1}{16}%
\left\vert \xi \right\vert ^{2}}\lesssim \exp \left( -c_{3}\left[ \delta
\left( \left\langle x\right\rangle -Mt\right) \right] ^{\frac{2}{p+1-\gamma }%
}\right)
\end{equation*}%
for some constant $c_{3}>0$ as $\delta \left( \left\langle x\right\rangle
-Mt\right) \leq 2\left\langle \xi \right\rangle ^{p+1-\gamma },$ we obtain
\begin{align*}
\left\vert I^{b}\right\vert & \lesssim \epsilon \int \left\langle \xi
\right\rangle ^{\gamma }\left\vert \mathrm{P}_{1}g\right\vert ^{2}d\xi
dx+\epsilon \int_{\delta \left( \left\langle x\right\rangle -Mt\right)
>0}\exp \left( -2c_{3}\left[ \delta \left( \left\langle x\right\rangle
-Mt\right) \right] ^{\frac{2}{p+1-\gamma }}\right) \left(
\sum_{j=0}^{4}b_{j}^{2}\right) dx \\
& \lesssim \epsilon \int \left\langle \xi \right\rangle ^{\gamma }\left\vert
\mathrm{P}_{1}g\right\vert ^{2}d\xi dx+\epsilon \left[ \int_{H_{+}}\left[
\delta \left( \left\langle x\right\rangle -Mt\right) \right] ^{\frac{\gamma
-1}{p+1-\gamma }}\left\vert \mathrm{P}_{0}g\right\vert ^{2}d\xi
dx+\int_{H_{0}\cup H_{-}}\left\vert \mathrm{P}_{0}g\right\vert ^{2}d\xi dx%
\right] .
\end{align*}

Finally, we split the integral
\begin{align*}
I^{c}& =\int_{\delta \left( \left\langle x\right\rangle -Mt\right)
>0}\int_{\delta \left( \left\langle x\right\rangle -Mt\right) \leq
2\left\langle \xi \right\rangle ^{p+1-\gamma }}\int_{\delta \left(
\left\langle x\right\rangle -Mt\right) \leq 2\left\langle \xi _{\ast
}\right\rangle ^{p+1-\gamma },\left\vert \xi _{\ast }\right\vert <\left\vert
\xi \right\vert } \\
& +\int_{\delta \left( \left\langle x\right\rangle -Mt\right)
>0}\int_{\delta \left( \left\langle x\right\rangle -Mt\right) \leq
2\left\langle \xi \right\rangle ^{p+1-\gamma }}\int_{\delta \left(
\left\langle x\right\rangle -Mt\right) \leq 2\left\langle \xi _{\ast
}\right\rangle ^{p+1-\gamma },\left\vert \xi _{\ast }\right\vert >\left\vert
\xi \right\vert } \\
& \equiv I_{1}^{c}+I_{2}^{c}.
\end{align*}%
Similar to $I^{b},$
\begin{equation*}
\left\vert I_{1}^{c}\right\vert \lesssim \epsilon \int \left\langle \xi
\right\rangle ^{\gamma }\left\vert \mathrm{P}_{1}g\right\vert ^{2}d\xi
dx+\epsilon \left[ \int_{H_{+}}\left[ \delta \left( \left\langle
x\right\rangle -Mt\right) \right] ^{\frac{\gamma -1}{p+1-\gamma }}\left\vert
\mathrm{P}_{0}g\right\vert ^{2}d\xi dx+\int_{H_{0}\cup H_{-}}\left\vert
\mathrm{P}_{0}g\right\vert ^{2}d\xi dx\right] .
\end{equation*}%
In view of $\left( \ref{B-epsilon}\right) ,$
\begin{align*}
& \int_{\delta \left( \left\langle x\right\rangle -Mt\right) \leq
2\left\langle \xi _{\ast }\right\rangle ^{p+1-\gamma },\text{ }\left\vert
\xi \right\vert <\left\vert \xi _{\ast }\right\vert }\left\vert k\left( \xi
,\xi _{\ast }\right) B_{\epsilon }\left( t,x,\xi ,\xi _{\ast }\right)
\mathrm{P}_{0}g\left( \xi _{\ast }\right) \right\vert d\xi _{\ast } \\
& \lesssim \epsilon \left( \sum_{j=0}^{4}b_{j}^{2}\right) ^{1/2}\int_{\delta
\left( \left\langle x\right\rangle -Mt\right) \leq 2\left\langle \xi _{\ast
}\right\rangle ^{p+1-\gamma }}\left\vert k\left( \xi ,\xi _{\ast }\right)
\right\vert \left\vert \xi _{\ast }\right\vert ^{2}e^{c_{1}\epsilon
\left\vert \xi _{\ast }\right\vert ^{2}}\left( \sum_{j=0}^{4}\left\vert \chi
_{j}\left( \xi _{\ast }\right) \right\vert ^{2}\right) ^{1/2}d\xi _{\ast } \\
& \lesssim \epsilon \left( \sum_{j=0}^{4}b_{j}^{2}\right) ^{1/2}\int_{\delta
\left( \left\langle x\right\rangle -Mt\right) \leq 2\left\langle \xi _{\ast
}\right\rangle ^{p+1-\gamma }}\left\vert k\left( \xi ,\xi _{\ast }\right)
\right\vert e^{-c^{\prime }\left\vert \xi _{\ast }\right\vert ^{2}}d\xi
_{\ast } \\
& \lesssim \epsilon \left( \sum_{j=0}^{4}b_{j}^{2}\right) ^{1/2}\exp \left( -%
\frac{c^{\prime \prime }}{2}\left[ \delta \left( \left\langle x\right\rangle
-Mt\right) \right] ^{\frac{2}{p+1-\gamma }}\right) \int_{\delta \left(
\left\langle x\right\rangle -Mt\right) \leq 2\left\langle \xi _{\ast
}\right\rangle ^{p+1-\gamma }}\left\vert k\left( \xi ,\xi _{\ast }\right)
\right\vert e^{-\frac{c^{\prime }}{2}|\xi _{\ast }|^{2}}d\xi _{\ast } \\
& \lesssim \epsilon \left( \sum_{j=0}^{4}b_{j}^{2}\right) ^{1/2}\exp \left( -%
\frac{c^{\prime \prime }}{2}\left[ \delta \left( \left\langle x\right\rangle
-Mt\right) \right] ^{\frac{2}{p+1-\gamma }}\right) e^{-\frac{c^{\prime }}{2}%
|\xi |^{2}}\quad \quad \quad \left( \text{by }\left( \ref{K-1}\right) \right)
\end{align*}%
for some constants $0<c^{\prime }<1/2$ and $c^{\prime \prime }>0,$ and then
by the Cauchy inequality,
\begin{align*}
\left\vert I_{2}^{c}\right\vert & \lesssim \epsilon \int \left\langle \xi
\right\rangle ^{\gamma }\left\vert \mathrm{P}_{1}g\right\vert ^{2}d\xi dx \\
& +\epsilon \int \left\langle \xi \right\rangle ^{-\gamma }e^{-\frac{%
c^{\prime }}{2}|\xi |^{2}}d\xi \cdot \int_{\delta \left( \left\langle
x\right\rangle -Mt\right) >0}\exp \left( -c^{\prime \prime }\left[ \delta
\left( \left\langle x\right\rangle -Mt\right) \right] ^{\frac{2}{p+1-\gamma }%
}\right) \sum_{j=0}^{4}b_{j}^{2}dx \\
& \lesssim \epsilon \int \left\langle \xi \right\rangle ^{\gamma }\left\vert
\mathrm{P}_{1}g\right\vert ^{2}d\xi dx+\epsilon \left[ \int_{H_{+}}\left[
\delta \left( \left\langle x\right\rangle -Mt\right) \right] ^{\frac{\gamma
-1}{p+1-\gamma }}\left\vert \mathrm{P}_{0}g\right\vert ^{2}d\xi
dx+\int_{H_{0}\cup H_{-}}\left\vert \mathrm{P}_{0}g\right\vert ^{2}d\xi dx%
\right] .
\end{align*}%
Consequently, we obtain%
\begin{align}
& \left\vert \int \left\langle \mathrm{P}_{1}g,\left( K_{\epsilon }-K\right)
\mathrm{P}_{0}g\right\rangle _{\xi }dx+\int \left\langle \mathrm{P}%
_{1}g,\left( K_{-\epsilon }-K\right) \mathrm{P}_{0}g\right\rangle _{\xi
}dx\right\vert  \notag \\
& \lesssim \epsilon \left[ \int \left\langle \xi \right\rangle ^{\gamma
}\left\vert \mathrm{P}_{1}g\right\vert ^{2}d\xi dx+\int_{H_{+}}\left[ \delta
\left( \left\langle x\right\rangle -Mt\right) \right] ^{\frac{\gamma -1}{%
p+1-\gamma }}\left\vert \mathrm{P}_{0}g\right\vert ^{2}d\xi
dx+\int_{H_{0}\cup H_{-}}\left\vert \mathrm{P}_{0}g\right\vert ^{2}d\xi dx%
\right] .  \label{K-P1P0}
\end{align}%
Combining $\left( \ref{K-P1P1}\right) ,$ $\left( \ref{K-P0P0}\right) $ and $%
\left( \ref{K-P1P0}\right) ,$ we get our result.
\end{proof}

Now, we are ready to get the weighted energy estimate of $\mathcal{R}^{(6)}$.

\begin{proposition}[Weighted energy for $\mathcal{R}^{(6)}$]
\label{weighted-remainder} Consider the weight
\begin{equation}
w\left( t,x,\xi \right) =\exp \left( \epsilon \rho \left( t,x,\xi \right)
/2\right) ,  \label{weight-function}
\end{equation}%
with%
\begin{align*}
\rho \left( t,x,\xi \right) & =5\left( \delta \left( \left\langle
x\right\rangle -Mt\right) \right) ^{\frac{p}{p+1-\gamma }}\left( 1-\chi
\left( \frac{\delta \left( \left\langle x\right\rangle -Mt\right) }{%
\left\langle \xi \right\rangle ^{p+1-\gamma }}\right) \right) \\
& +\left[ \left( 1-\chi \left( \frac{\delta \left( \left\langle
x\right\rangle -Mt\right) }{\left\langle \xi \right\rangle ^{p+1-\gamma }}%
\right) \right) \left[ \delta \left( \left\langle x\right\rangle -Mt\right) %
\right] \left\langle \xi \right\rangle ^{\gamma -1}+3\left\langle \xi
\right\rangle ^{p}\right] \chi \left( \frac{\delta \left( \left\langle
x\right\rangle -Mt\right) }{\left\langle \xi \right\rangle ^{p+1-\gamma }}%
\right) ,
\end{align*}%
where $\epsilon ,\ \delta >0\ $are sufficiently small, $M>0$ sufficiently
large, and $0<p\leq 2.$ Then we have%
\begin{equation*}
\left\Vert w\mathcal{R}^{(6)}\right\Vert _{H_{x}^{2}L_{\xi }^{2}}\lesssim
\{t^{5}\wedge t\}\left\Vert f_{0}\right\Vert _{L^{2}(\mu )},\ \ \ \ \ 0\leq
\gamma <1,
\end{equation*}%
and%
\begin{equation*}
\left\Vert w\mathcal{R}^{(6)}\right\Vert _{H_{x}^{2}L_{\xi }^{2}}\lesssim
t^{5}\left( 1+t\right) ^{3}\Vert f_{0}\Vert _{L^{2}(\mu )},\ \ \ -2<\gamma
<0.
\end{equation*}
\end{proposition}

\begin{proof}
Let $u=w\mathcal{R}^{(6)}=e^{\frac{\epsilon \rho }{2}}\mathcal{R}^{(6)}$,
and then $u$ solves the equation
\begin{equation*}
\partial _{t}u+\xi \cdot \nabla _{x}u-\frac{\epsilon }{2}(\partial _{t}\rho
+\xi \cdot \nabla _{x}\rho )u-e^{\frac{\epsilon \rho }{2}}L\left( e^{-\frac{%
\epsilon \rho }{2}}u\right) =e^{\frac{\epsilon \rho }{2}}Kh^{(6)}\,.
\end{equation*}%
The energy estimate gives
\begin{align*}
& \quad \frac{1}{2}\frac{d}{dt}\int_{{\mathbb{R}}^{3}}\left\langle
u,u\right\rangle _{\xi }dx-\int_{{\mathbb{R}}^{3}}\left\langle u,e^{\frac{%
\epsilon \rho }{2}}Kh^{(6)}\right\rangle _{\xi }dx \\
& =\int_{{\mathbb{R}}^{3}}\frac{\epsilon }{2}\left\langle u,(\partial
_{t}\rho +\xi \cdot \nabla _{x}\rho )u\right\rangle _{\xi }dx+\int_{{\mathbb{%
R}}^{3}}\left\langle u,e^{\frac{\epsilon \rho }{2}}L\left( e^{-\frac{%
\epsilon \rho }{2}}u\right) \right\rangle _{\xi }dx\,.
\end{align*}%
In view of Lemma \ref{K-epsilon},
\begin{align*}
\int_{{\mathbb{R}}^{3}}\left\langle u,e^{\frac{\epsilon \rho }{2}}L\left(
e^{-\frac{\epsilon \rho }{2}}u\right) \right\rangle _{\xi }dx& \leq -\mu
\int_{{\mathbb{R}}^{3}}|\left\langle \xi \right\rangle ^{\frac{\gamma }{2}}%
\mathrm{P}_{1}u|_{L_{\xi }^{2}}^{2}dx \\
& +C_{1}\epsilon \left[ \int_{H_{+}}\left[ \delta \left( \left\langle
x\right\rangle -Mt\right) \right] ^{\frac{\gamma -1}{p+1-\gamma }}\left\vert
\mathrm{P}_{0}u\right\vert ^{2}d\xi dx+\int_{H_{0}\cup H_{-}}\left\vert
\mathrm{P}_{0}u\right\vert ^{2}d\xi dx\right] ,
\end{align*}%
for some constants $\mu >0\ $and $C_{1}>0.$ One can easily check that
\begin{align*}
\partial _{t}\rho & =-\delta M\left\langle \xi \right\rangle ^{\gamma
-1}\left( \frac{5p}{p+1-\gamma }\left[ \left( \delta \left( \left\langle
x\right\rangle -Mt\right) \right) \left\langle \xi \right\rangle ^{\gamma
-p-1}\right] ^{\frac{\gamma -1}{p+1-\gamma }}\left( 1-\chi \right) +\chi
(1-\chi )\right) \\
& \quad +\delta M\left( 5\left[ \left( \delta \left( \left\langle
x\right\rangle -Mt\right) \right) \left\langle \xi \right\rangle ^{\gamma
-p-1}\right] ^{\frac{p}{p+1-\gamma }}-(1-2\chi )\left[ \left( \delta \left(
\left\langle x\right\rangle -Mt\right) \right) \left\langle \xi
\right\rangle ^{\gamma -p-1}\right] -3\right) \left\langle \xi \right\rangle
^{\gamma -1}\chi ^{\prime }\,\leq 0
\end{align*}%
(the constants $5$ and $3$ are chosen intentionally such that the quantity
in the latter bracket is nonnegative on $H_{0}$), and
\begin{align*}
\nabla _{x}\rho & =\delta \left( \nabla _{x}\left\langle x\right\rangle
\right) \left\langle \xi \right\rangle ^{\gamma -1}\left( \frac{5p}{%
p+1-\gamma }\left[ \left( \delta \left( \left\langle x\right\rangle
-Mt\right) \right) \left\langle \xi \right\rangle ^{\gamma -p-1}\right] ^{%
\frac{\gamma -1}{p+1-\gamma }}\left( 1-\chi \right) +\chi (1-\chi )\right) \\
& \quad -\delta \left( \nabla _{x}\left\langle x\right\rangle \right) \left(
5\left[ \left( \delta \left( \left\langle x\right\rangle -Mt\right) \right)
\left\langle \xi \right\rangle ^{\gamma -p-1}\right] ^{\frac{p}{p+1-\gamma }%
}-(1-2\chi )\left[ \left( \delta \left( \left\langle x\right\rangle
-Mt\right) \right) \left\langle \xi \right\rangle ^{\gamma -p-1}\right]
-3\right) \left\langle \xi \right\rangle ^{\gamma -1}\chi ^{\prime }\,.
\end{align*}%
Hence,
\begin{equation*}
\partial _{t}\rho =\xi \cdot \nabla _{x}\rho =0\quad \text{on}\quad H_{-}\,,
\end{equation*}%
\begin{equation*}
|\partial _{t}\rho |\lesssim \delta M\left\langle \xi \right\rangle ^{\gamma
-1}\quad \text{and}\quad |\xi \cdot \nabla _{x}\rho |\lesssim \delta
\left\langle \xi \right\rangle ^{\gamma }\ \ \text{on}\quad H_{0},
\end{equation*}%
and we have
\begin{equation*}
\partial _{t}\rho =-\frac{5p\delta M}{p+1-\gamma }\left[ \left( \delta
\left( \left\langle x\right\rangle -Mt\right) \right) \right] ^{\frac{\gamma
-1}{p+1-\gamma }},
\end{equation*}%
\begin{equation*}
\xi \cdot \nabla _{x}\rho =\frac{5p\delta }{p+1-\gamma }\frac{\xi \cdot x}{%
\left\langle x\right\rangle }\left[ \left( \delta \left( \left\langle
x\right\rangle -Mt\right) \right) \right] ^{\frac{\gamma -1}{p+1-\gamma }},
\end{equation*}%
on $H_{+}.$ Direct calculation together with the Cauchy inequality show that
\begin{equation*}
\epsilon \left\vert \int_{{\mathbb{R}}^{3}}\left\langle u,\left( \xi \cdot
\nabla _{x}\rho \right) u\right\rangle _{\xi }dx\right\vert \leq
C_{2}\epsilon \delta \left[
\begin{array}{l}
\int_{{\mathbb{R}}^{3}}|\left\langle \xi \right\rangle ^{\frac{\gamma }{2}}%
\mathrm{P}_{1}u|_{L_{\xi }^{2}}^{2}dx \\
+\int_{H_{+}}\left[ \left( \delta \left( \left\langle x\right\rangle
-Mt\right) \right) \right] ^{\frac{\gamma -1}{p+1-\gamma }}|\mathrm{P}%
_{0}u|^{2}d\xi dx+\int_{H_{0}}|\mathrm{P}_{0}u|^{2}d\xi dx%
\end{array}%
\right] ,
\end{equation*}%
and
\begin{align*}
\epsilon \int_{{\mathbb{R}}^{3}}\left\langle u,\left( \partial _{t}\rho
\right) u\right\rangle _{\xi }dx& \leq \epsilon \delta MC_{3}\int_{{\mathbb{R%
}}^{3}}|\left\langle \xi \right\rangle ^{\frac{\gamma }{2}}\mathrm{P}%
_{1}u|_{L_{\xi }^{2}}^{2}dx \\
& -\epsilon \delta MC_{4}\int_{H_{+}}\left[ \left( \delta \left(
\left\langle x\right\rangle -Mt\right) \right) \right] ^{\frac{\gamma -1}{%
p+1-\gamma }}|\mathrm{P}_{0}u|^{2}d\xi dx+\epsilon \delta MC_{5}\int_{H_{0}}|%
\mathrm{P}_{0}u|^{2}d\xi dx\,.
\end{align*}%
In conclusion, we get
\begin{align*}
& \quad \frac{1}{2}\frac{d}{dt}\int_{{\mathbb{R}}^{3}}\left\langle
u,u\right\rangle _{\xi }dx-\int_{{\mathbb{R}}^{3}}\left\langle
u,wKh^{(6)}\right\rangle _{\xi }dx \\
& \leq -(\mu -\epsilon \delta C_{2}-\epsilon \delta MC_{3})\int_{{\mathbb{R}}%
^{3}}|\left\langle \xi \right\rangle ^{\frac{\gamma }{2}}\mathrm{P}%
_{1}u|_{L_{\xi }^{2}}^{2}dx \\
& -\epsilon \left( \delta MC_{4}-\delta C_{2}-C_{1}\right) \int_{H_{+}}\left[
\delta (\left\langle x\right\rangle -Mt)\right] ^{\frac{r-1}{p+1-\gamma }}|%
\mathrm{P}_{0}u|^{2}d\xi dx \\
& \quad +\epsilon (\delta C_{2}+\delta MC_{5})\int_{H_{0}}|\mathrm{P}%
_{0}u|^{2}d\xi dx+\epsilon C_{1}\int_{H_{-}}|\mathrm{P}_{0}u|^{2}d\xi dx.
\end{align*}%
Choosing $\delta $, $\epsilon >0$ small and $M>0$ large enough, we have
\begin{equation*}
\frac{d}{dt}\Vert u\Vert _{L^{2}}^{2}\lesssim \Vert u\Vert _{L^{2}}\Vert
wKh^{(6)}\Vert _{L^{2}}+\Vert \mathrm{P}_{0}u\Vert _{L^{2}}^{2}\,\lesssim
\Vert u\Vert _{L^{2}}\Vert wKh^{(6)}\Vert _{L^{2}}+\Vert u\Vert
_{L^{2}}\,\Vert \mathcal{R}^{(6)}\Vert _{L^{2}}.
\end{equation*}%
Moreover, since $\partial _{t}\rho \leq 0$, the weight $w$ is decreasing in $%
t,$ so that%
\begin{equation*}
\Vert wKh^{(6)}\Vert _{L^{2}}\leq \Vert \mu ^{1/2}Kh^{(6)}\Vert
_{L^{2}}=\Vert Kh^{(6)}\Vert _{L^{2}\left( \mu \right) }\lesssim \Vert
h^{(6)}\Vert _{L^{2}\left( \mu \right) },
\end{equation*}%
due to $\left( \ref{K-3}\right) .$ It implies
\begin{equation*}
\frac{d}{dt}\Vert u\Vert _{L^{2}}\lesssim \Vert h^{(6)}\Vert _{L^{2}\left(
\mu \right) }+\Vert \mathcal{R}^{(6)}\Vert _{L^{2}}.
\end{equation*}%
For the $x$-derivative estimate, we only need to control the commutator
terms:
\begin{equation}
\epsilon \int_{{\mathbb{R}}^{3}}\left\langle \partial _{x_{i}}u,\partial
_{x_{i}}\left( \partial _{t}\rho +\xi \cdot \nabla _{x}\rho \right)
u\right\rangle _{\xi }dx,  \label{comm1}
\end{equation}%
\begin{equation}
\epsilon \int_{{\mathbb{R}}^{3}}\left\langle \partial _{x_{i}}u,e^{\frac{%
\epsilon \rho }{2}}K\left( \partial _{x_{i}}\rho e^{-\frac{\epsilon \rho }{2}%
}u\right) \right\rangle _{\xi }dx\,,  \label{comm2}
\end{equation}%
\begin{equation}
\epsilon \int_{{\mathbb{R}}^{3}}\left\langle \partial _{x_{i}}u,\partial
_{x_{i}}\rho e^{\frac{\epsilon \rho }{2}}Ke^{-\frac{\epsilon \rho }{2}%
}u\right\rangle _{\xi }dx=\epsilon \int_{{\mathbb{R}}^{3}}\left\langle
\partial _{x_{i}}\rho \partial _{x_{i}}u,e^{\frac{\epsilon \rho }{2}}Ke^{-%
\frac{\epsilon \rho }{2}}u\right\rangle _{\xi }dx,  \label{comm3}
\end{equation}%
and
\begin{equation}
\int_{{\mathbb{R}}^{3}}\left\langle \partial _{x_{i}}u,\partial
_{x_{i}}\left( wKh^{(6)}\right) \right\rangle _{\xi }dx.  \label{comm4}
\end{equation}%
It is obvious that the decay of $\partial _{x_{i}}\left( \partial _{t}\rho
+\xi \cdot \nabla _{x}\rho \right) $ is faster than $\left( \partial
_{t}\rho +\xi \cdot \nabla _{x}\rho \right) $, hence the first term (\ref%
{comm1}) is easy to control. From $\left( \ref{Weighted-K-2}\right) $ and
the fact that $\left\vert \partial _{x_{i}}\rho \right\vert \lesssim
\left\langle \xi \right\rangle ^{\gamma -1},$ it follows that for $0\leq
\gamma <1,$
\begin{eqnarray*}
\epsilon \left\vert \int_{{\mathbb{R}}^{3}}\left\langle \partial
_{x_{i}}u,e^{\frac{\epsilon \rho }{2}}K\left( \partial _{x_{i}}\rho e^{-%
\frac{\epsilon \rho }{2}}u\right) \right\rangle _{\xi }dx\right\vert
&\lesssim &\epsilon \left( \left\Vert \partial _{x_{i}}u\right\Vert
_{L^{2}}^{2}+\left\Vert \left\langle \xi \right\rangle ^{\frac{\gamma }{2}}%
\mathrm{P}_{1}u\right\Vert _{L^{2}}^{2}+\left\Vert \left\langle \xi
\right\rangle ^{\frac{\gamma }{2}}\mathrm{P}_{0}u\right\Vert
_{L^{2}}^{2}\right) \\
&\lesssim &\epsilon \left( \left\Vert \mathrm{P}_{0}\partial
_{x_{i}}u\right\Vert _{L^{2}}^{2}+\left\Vert \left\langle \xi \right\rangle
^{\frac{\gamma }{2}}\mathrm{P}_{1}u\right\Vert _{L^{2}}^{2}+\left\Vert
\mathrm{P}_{0}u\right\Vert _{L^{2}}^{2}\right) ,
\end{eqnarray*}%
\begin{eqnarray*}
\epsilon \left\vert \int_{{\mathbb{R}}^{3}}\left\langle \partial
_{x_{i}}\rho \partial _{x_{i}}u,e^{\frac{\epsilon \rho }{2}}Ke^{-\frac{%
\epsilon \rho }{2}}u\right\rangle _{\xi }dx\right\vert &\lesssim &\epsilon
\left( \left\Vert \left\langle \xi \right\rangle ^{\frac{\gamma }{2}}\mathrm{%
P}_{1}\partial _{x_{i}}u\right\Vert _{L^{2}}^{2}+\left\Vert \left\langle \xi
\right\rangle ^{\frac{\gamma }{2}}\mathrm{P}_{0}\partial
_{x_{i}}u\right\Vert _{L^{2}}^{2}+\left\Vert u\right\Vert _{L^{2}}^{2}\right)
\\
&\lesssim &\epsilon \left( \left\Vert \left\langle \xi \right\rangle ^{\frac{%
\gamma }{2}}\mathrm{P}_{1}\partial _{x_{i}}u\right\Vert
_{L^{2}}^{2}+\left\Vert \mathrm{P}_{0}\partial _{x_{i}}u\right\Vert
_{L^{2}}^{2}+\left\Vert \left\langle \xi \right\rangle ^{\frac{\gamma }{2}}%
\mathrm{P}_{1}u\right\Vert _{L^{2}}^{2}+\left\Vert \mathrm{P}%
_{0}u\right\Vert _{L^{2}}^{2}\right) ,
\end{eqnarray*}%
and that for $-2<\gamma <1,$%
\begin{eqnarray*}
\epsilon \left\vert \int_{{\mathbb{R}}^{3}}\left\langle \partial
_{x_{i}}u,e^{\frac{\epsilon \rho }{2}}K\left( \partial _{x_{i}}\rho e^{-%
\frac{\epsilon \rho }{2}}u\right) \right\rangle _{\xi }dx\right\vert
&\lesssim &\epsilon \left( \left\Vert \partial _{x_{i}}u\right\Vert
_{L_{\sigma }^{2}}^{2}+\left\Vert \left\langle \xi \right\rangle ^{\frac{%
\gamma }{2}}\mathrm{P}_{1}u\right\Vert _{L_{\sigma }^{2}}^{2}+\left\Vert
\left\langle \xi \right\rangle ^{\frac{\gamma }{2}}\mathrm{P}%
_{0}u\right\Vert _{L_{\sigma }^{2}}^{2}\right) \\
&\lesssim &\epsilon \left( \left\Vert \mathrm{P}_{0}\partial
_{x_{i}}u\right\Vert _{L^{2}}^{2}+\left\Vert \left\langle \xi \right\rangle
^{\frac{\gamma }{2}}\mathrm{P}_{1}u\right\Vert _{L^{2}}^{2}+\left\Vert
\mathrm{P}_{0}u\right\Vert _{L^{2}}^{2}\right) ,
\end{eqnarray*}%
\begin{eqnarray*}
\epsilon \left\vert \int_{{\mathbb{R}}^{3}}\left\langle \partial
_{x_{i}}\rho \partial _{x_{i}}u,e^{\frac{\epsilon \rho }{2}}Ke^{-\frac{%
\epsilon \rho }{2}}u\right\rangle _{\xi }dx\right\vert &\lesssim &\epsilon
\left( \left\Vert \left\langle \xi \right\rangle ^{\frac{\gamma }{2}}\mathrm{%
P}_{1}\partial _{x_{i}}u\right\Vert _{L_{\sigma }^{2}}^{2}+\left\Vert
\left\langle \xi \right\rangle ^{\frac{\gamma }{2}}\mathrm{P}_{0}\partial
_{x_{i}}u\right\Vert _{L_{\sigma }^{2}}^{2}+\left\Vert u\right\Vert
_{L_{\sigma }^{2}}^{2}\right) \\
&\lesssim &\epsilon \left( \left\Vert \left\langle \xi \right\rangle ^{\frac{%
\gamma }{2}}\mathrm{P}_{1}\partial _{x_{i}}u\right\Vert
_{L^{2}}^{2}+\left\Vert \mathrm{P}_{0}\partial _{x_{i}}u\right\Vert
_{L^{2}}^{2}+\left\Vert \left\langle \xi \right\rangle ^{\frac{\gamma }{2}}%
\mathrm{P}_{1}u\right\Vert _{L^{2}}^{2}+\left\Vert \mathrm{P}%
_{0}u\right\Vert _{L^{2}}^{2}\right) .
\end{eqnarray*}%
On the other hand,%
\begin{equation*}
\left\vert \int_{{\mathbb{R}}^{3}}\left\langle \partial _{x_{i}}u,\partial
_{x_{i}}\left( wKh^{(6)}\right) \right\rangle _{\xi }dx\right\vert \lesssim
\left\Vert \partial _{x_{i}}u\right\Vert _{L^{2}}\left\Vert \partial
_{x_{i}}\left( wKh^{(6)}\right) \right\Vert _{L^{2}}.
\end{equation*}%
The second derivative estimate is similar and hence we omit the details. We
then deduce that
\begin{align*}
\frac{d}{dt}\Vert u\Vert _{H_{x}^{2}L_{\xi }^{2}}^{2}& \lesssim \Vert u\Vert
_{H_{x}^{2}L_{\xi }^{2}}\Vert wKh^{(6)}\Vert _{H_{x}^{2}L_{\xi }^{2}}+\Vert
\mathrm{P}_{0}u\Vert _{H_{x}^{2}L_{\xi }^{2}} \\
& \lesssim \Vert u\Vert _{H_{x}^{2}L_{\xi }^{2}}\Vert wKh^{(6)}\Vert
_{H_{x}^{2}L_{\xi }^{2}}+\Vert u\Vert _{H_{x}^{2}L_{\xi }^{2}}\Vert \mathcal{%
R}^{(6)}\Vert _{H_{x}^{2}L_{\xi }^{2}} \\
& \lesssim \Vert u\Vert _{H_{x}^{2}L_{\xi }^{2}}\left( \Vert h^{(6)}\Vert
_{H_{x}^{2}L_{\xi }^{2}\left( \mu \right) }+\Vert \mathcal{R}^{(6)}\Vert
_{H_{x}^{2}L_{\xi }^{2}}\right) ,
\end{align*}%
the last inequality holds since $\left\vert \partial _{x_{i}}\rho
\right\vert \lesssim \left\langle \xi \right\rangle ^{\gamma -1}$ and the
weight $w$ is decreasing in $t.$ It follows that
\begin{equation*}
\frac{d}{dt}\Vert u\Vert _{H_{x}^{2}L_{\xi }^{2}}\lesssim \Vert h^{(6)}\Vert
_{H_{x}^{2}L_{\xi }^{2}\left( \mu \right) }+\Vert \mathcal{R}^{(6)}\Vert
_{H_{x}^{2}L_{\xi }^{2}}\,.
\end{equation*}%
In view of Lemma \ref{Regularity} and (\ref{R^(6)-1}),
\begin{equation*}
\frac{d}{dt}\Vert u\Vert _{H_{x}^{2}L_{\xi }^{2}}\lesssim \left\{
\begin{array}{ll}
\{t^{4}\wedge 1\}\left(\Vert f_{0}\Vert _{L^{2}(\mu )}+\Vert f_{0}\Vert
_{L^{2}}\right),%
\vspace{3mm}
& 0\leq \gamma <1, \\
t^{4}\left( 1+t\right) ^{3}\left( \Vert f_{0}\Vert _{L^{2}(\mu )}+\Vert
f_{0}\Vert _{L^{2}}\right) , & -2<\gamma <0.%
\end{array}%
\right.
\end{equation*}%
This completes the proof of the proposition.
\end{proof}


Through Proposition \ref{weighted-remainder} and the Sobolev inequality, we
will establish the pointwise estimate for $\mathcal{R}^{(6)}$ in the
following. Combining this with the wave part $W^{(6)}$ (see Lemma \ref%
{pointwise-h}), we complete the wave structure of the solution outside the
finite Mach number region.

\begin{proposition}
Let $\mathcal{R}^{(6)}$ be the remainder part of the linearized Boltzmann
equation \eqref{bot.1.d} with $-2<\gamma <1,$ and $0<p\leq 2.\ $There exists
a positive constant $M$ such that for $\left\langle x\right\rangle >2Mt$, we
have
\begin{equation}
\left\vert \mathcal{R}^{(6)}(t,x,\cdot )\right\vert _{L_{\xi }^{2}}\leq
Ct^{5}e^{-c_{\epsilon }(\left\langle x\right\rangle +t)^{\frac{p}{p+1-\gamma
}}} \left\Vert f_{0}\right\Vert _{L^{2}\left( e^{\epsilon \left\vert \xi
\right\vert ^{p}}\right) }\,,
\end{equation}%
where the constant $\epsilon >0$ is sufficiently small and $C,c_{\epsilon }$
are some positive constants.
\end{proposition}

\begin{proof}
Let $w$ be the weight function defined as \eqref{weight}. Observe that for $%
\left\langle x\right\rangle >2Mt$,
\begin{equation*}
\rho (t,x,\xi )\gtrsim \left( \delta (\left\langle x\right\rangle
-Mt)\right) ^{\frac{p}{p+1-\gamma }}.
\end{equation*}%
Applying Proposition \ref{weighted-remainder}, it follows from the Sobolev
inequality \cite[Proposition 3.8]{[Taylor]} that%
\begin{eqnarray*}
e^{\epsilon \left( \delta (\left\langle x\right\rangle -Mt)\right) ^{\frac{p%
}{p+1-\gamma }}}\left\vert \mathcal{R}^{(6)}(t,x,\cdot )\right\vert _{L_{\xi
}^{2}} &\leq &\left\vert w\mathcal{R}^{(6)}\right\vert _{L_{\xi }^{2}}\leq
\left\Vert w\mathcal{R}^{(6)}\right\Vert _{L_{\xi }^{2}L_{x}^{\infty }} \\
&\lesssim &\left\Vert \nabla _{x}^{2}\left( w\mathcal{R}^{(6)}\right)
\right\Vert _{L^{2}}^{1/2}\left\Vert \nabla _{x}\left( w\mathcal{R}%
^{(6)}\right) \right\Vert _{L^{2}}^{1/2}\lesssim \left\Vert w\mathcal{R}%
^{(6)}\right\Vert _{H_{x}^{2}L_{\xi }^{2}} \\
&\lesssim &\left\{
\begin{array}{ll}
\{t^{5}\wedge t\}\left\Vert f_{0}\right\Vert _{L^{2}(\mu )},%
\vspace {3mm}
& 0\leq \gamma <1, \\
t^{5}\left( 1+t\right) ^{3}\left\Vert f_{0}\right\Vert _{L^{2}(\mu )}, &
-2<\gamma <0.%
\end{array}%
\right.
\end{eqnarray*}%
Here $\epsilon >0$ can be chosen as small as we want. Note that for $%
\left\langle x\right\rangle >2Mt$,
\begin{equation*}
\left\langle x\right\rangle -Mt>\frac{\left\langle x\right\rangle }{3}+\frac{%
Mt}{3},
\end{equation*}%
and
\begin{equation*}
\left\Vert f_{0}\right\Vert _{L^{2}(\mu )}\lesssim \Vert f_{0}\Vert
_{L^{2}(e^{\epsilon |\xi |^{p}})},
\end{equation*}%
due to the fact that $f_{0}$ has compact support in variable $x.$ Therefore
there exist positive constants $C$ and $c_{\epsilon }$ such that
\begin{equation}
\left\vert \mathcal{R}^{(6)}(t,x,\cdot )\right\vert _{L_{\xi }^{2}}\leq
Ct^{5}e^{-c_{\epsilon }(\left\langle x\right\rangle +t)^{\frac{p}{p+1-\gamma
}}}\Vert f_{0}\Vert _{L^{2}(e^{\epsilon |\xi |^{p}})}.  \label{remainder}
\end{equation}
\end{proof}


\begin{thebibliography}{99}
\bibitem{[Adams]} R. Adams and J. Fournier, Sobolev spaces, vol. 140 (2003),
Academic Press; 2 edition.

\bibitem{[Caflisch]} R. Caflisch, The Boltzmann equation with a soft
potential. I. Linear, spatially homogeneous, Comm. Math. Phys. 74(1980),
71-95.

\bibitem{[GolsePoupand]} F. Golse and F. Poupaud, Stationary solutions of
the linearized Boltzmann equation in a half-space, Math. Methods Appl. Sci.
11(1989), 483-502.

\bibitem{[cc]} C.C. Chen, T.P. Liu and T. Yang, Existence of boundary layer
solutions to the Boltzmann equation, Anal. Appl.(Singap.) 2(2004), 337-363.

\bibitem{[Ellis]} R. Ellis and M. Pinsky, The first and second fluid
approximations to the linearized Boltzmann equation, J. Math. Pure. App.,
54(1975), 125--156.

\bibitem{[Glassey]} R. Glassey, The Cauchy problem in kinetic theory, SIAM,
Philadelphia, 1996.

\bibitem{[Grad]} H. Grad, Asymptotic theory of the Boltzmann equation,
Rarefied Gas Dynamics, J. A. Laurmann, Ed. 1, 26, pp.26--59 Academic Press,
New York, 1963.

\bibitem{[Gualdani]} M.P. Gualdani, S. Mischler and C. Mouhot, Factorization
of non-symmetric operators and exponential H-theorem, to appear as a
M\'emoire de la Soci\'et\'e Math\'ematique de France.

\bibitem{[Kawashima]} S. Kawashima, The Boltzmann equation and thirteen
moments, Japan J. Appl. Math., 7(1990), 301-320.


\bibitem{[LeeLiuYu]} M.Y. Lee, T.P. Liu and S.H. Yu, Large time behavier of
solutions for the Boltzmann equation with hard potentials, Commun. Math.
Phys., 269(2007), 17--37.

\bibitem{[LiuYu]} T.P. Liu and S.H. Yu, The Green function and large time
behavier of solutions for the one-dimensional Boltzmann equation, Commun.
Pure App. Math., 57(2004), 1543--1608.

\bibitem{[LiuYu2]} T.P. Liu and S.H. Yu, Green's function of Boltzmann
equation, 3-D waves, Bull. Inst. Math. Acad. Sin. (N.S.), 1(2006), 1-78.

\bibitem{[LiuYu1]} T.P. Liu and S.H. Yu, Solving Boltzmann equation, Part I
: Green's function, Bull. Inst. Math. Acad. Sin. (N.S.), 6(2011), 151-243.

\bibitem{[Strain]} R.M. Strain, Optimal time decay of the non cut-off
Boltzmann equation in the whole space. Kinet. Relat. Models, 5(2012),
583-613.

\bibitem{[Strain-Guo]} R.M. Strain and Y. Guo, Exponential decay for soft
potentials near Maxwellian, Arch. Ration. Mech. Anal., 187(2008), 287-339.

\bibitem{[Taylor]} Michael E. Taylor, Partial differential equations, III,
Applied Mathematical Sciences, vol. 117, Springer-Verlag, New York, 1997.
Nonlinear equations; Corrected reprint of the 1996 original.

\bibitem{[LWW1]} H.T. Wang and K.-C. Wu, Solving linearized
Landau equation pointwisely, submitted, arXiv:1709.00839.

\bibitem{[Wu]} K.-C. Wu, Pointwise Behavior of the Linearized Boltzmann
Equation on a Torus, SIAM J. Math. Anal., 46 (2014), 639--656.
\end{thebibliography}
\end{document}